\documentclass{amsart}
\begin{document}


\newtheorem{theo}{Theorem}
\newtheorem{lem}{Lemma}
\newtheorem{question}{Question}
\newtheorem{cor}{Corollary}
\newtheorem{prop}{Proposition}
\newtheorem{Question}{Question}
\newtheorem{prb}{Problem}
\newtheorem{definition}{Definition}
\newtheorem{example}{Example}
\theoremstyle{definition}
\newtheorem{dfn}{Definition}
\newtheorem{rem}{Remark}

\newcommand{\sep}{, }
\newcommand{\cs}{\mathrm{cs}}
\newcommand{\sbr}{\mathrm{sb}}

\title[On groups with countable
$\mathrm{cs}^*$-character]{The topological structure of (homogeneous) spaces and groups with
countable $\bold{cs}^*$-character}
\author[T.Banakh, L.Zdomskyy]{Taras Banakh and Lubomyr Zdomsky}

\begin{abstract}
In this paper we introduce and study three new cardinal
topological invariants called the $\mathrm{cs}^*$-,
$\mathrm{cs}$-, and $\mathrm{sb}$-characters. The class of
topological spaces with countable $\mathrm{cs}^*$-character is
closed under many topological operations and contains all
$\aleph$-spaces and all spaces with point-countable
$\mathrm{cs}^*$-network. Our principal result states that each
non-metrizable sequential topological group with countable
$\mathrm{cs}^*$-character has countable pseudo-character and
contains an open $k_\omega$-subgroup. This result is specific for
topological groups: under Martin Axiom there exists a sequential
topologically homogeneous $k_\omega$-space $X$ with
$\aleph_0=\mathrm{cs}^*_\chi(X)<\psi(X)$.
\end{abstract}


\maketitle 

\section*{Introduction}

In this paper we introduce and study three new local cardinal
invariants of topological spaces called the
$\mathrm{sb}$-character, the $\mathrm{cs}$-character and
$\mathrm{cs}^*$-character, and describe the structure of
sequential topological groups with countable
$\mathrm{cs}^*$-character. All these characters are based on the
notion of a {\em network} at a point $x$ of a topological space
$X$, under which we understand a collection $\mathcal N$ of
subsets of $X$ such that for any neighborhood $U\subset X$ of $x$
there is an element $N\in\mathcal N$ with $x\in N\subset U$, see
\cite{Lin}.

A subset $B$ of a topological space $X$ is called a {\em
sequential barrier} at a point $x\in X$ if for any sequence
$(x_n)_{n\in\omega}\subset X$ convergent to $x$, there is
$m\in\omega$ such that $x_n\in B$ for all $n\ge m$, see
\cite{Lin}. It is clear that each neighborhood of a point $x\in X$
is a sequential barrier for $x$ while the converse in true for
Fr\'echet-Urysohn spaces.

Under a {\em $\mathrm{sb}$-network\/} at a point $x$ of a
topological space $X$ we shall understand a network at $x$
consisting of sequential barriers at $x$. In other words, a
collection $\mathcal N$ of subsets of $X$ is a
$\mathrm{sb}$-network at $x$ if for any neighborhood $U$ of $x$
there is an element $N$ of $\mathcal N$ such that for any sequence
$(x_n)\subset X$ convergent to $x$ the set $N$ contains almost all
elements of $(x_n)$. Changing two quantifiers in this definition
by their places we get a definition of a $\mathrm{cs}$-network at
$x$.

Namely, we define a family $\mathcal N$ of subsets of a
topological space $X$ to be a {\em $\mathrm{cs}$-network} (resp. a
{\em $\mathrm{cs}^*$-network}) at a point $x\in X$ if for any
neighborhood $U\subset X$ of $x$ and any sequence $(x_n)\subset X$
convergent to $x$ there is an element $N\in\mathcal N$ such that
$N\subset U$ and $N$ contains almost all (resp. infinitely many)
members of the sequence $(x_n)$. A family $\mathcal N$ of subsets
of a topological space $X$ is called a {\em $\mathrm{cs}$-network}
(resp. {\em $\mathrm{cs}^*$-network}) if it is a
$\mathrm{cs}$-network (resp. $\mathrm{cs}^*$-network) at each
point $x\in X$, see \cite{Nag}.

The smallest size $|\mathcal N|$ of an $\mathrm{sb}$-network
(resp. $\mathrm{cs}$-network, $\mathrm{cs}^*$-network) $\mathcal
N$ at a point $x\in X$ is called the {\em $\mathrm{sb}$-character}
(resp. {\em $\mathrm{cs}$-character, $\mathrm{cs}^*$-character})
of $X$ at the point $x$ and is denoted by $\mathrm{sb}_\chi(X,x)$
(resp. $\mathrm{cs}_\chi(X,x)$, $\mathrm{cs}^*_\chi(X,x)$). The
cardinals $\mathrm{sb}_\chi(X)=\sup_{x\in
X}\mathrm{sb}_\chi(X,x)$, $\mathrm{cs}_\chi(X)=\sup_{x\in
X}\mathrm{cs}_\chi(X,x)$ and $\mathrm{cs}^*_\chi(X)=\sup_{x\in X}
\mathrm{cs}^*_\chi(X,x)$ are called the {\em
$\mathrm{sb}$-character}, {\em $\cs$-character} and {\em
$\mathrm{cs}^*$-character} of the topological space $X$,
respectively. For the empty topological space $X=\emptyset$ we put
$\mathrm{sb}_\chi(X)=\mathrm{cs}_\chi(X)=\mathrm{cs}^*_\chi(X)=1$.

In the sequel we shall say that a topological space $X$ has {\em
countable $\mathrm{sb}$-character (resp. $\mathrm{cs}$-,
$\mathrm{cs}^*$-character})  if $\mathrm{sb}_\chi(X)\le\aleph_0$
(resp. $\mathrm{cs}_\chi(X)\le\aleph_0$,
$\mathrm{cs}^*_\chi(X)\le\aleph_0$). In should be mentioned that
under different names topological spaces with countable
$\mathrm{sb}$- or $\mathrm{cs}$-character have already occured in
topological literature. In particular, a topological space has
countable $\mathrm{cs}$-character if and only if it is
$csf$-countable in the sense of \cite{Lin};
 a (sequential) space $X$ has countable
$\mathrm{sb}$-character if and only if it is universally
$csf$-countable in the sense of \cite{Lin} (if and only if it is
weakly first-countable in the sense of \cite{Ar} if and only if it
is 0-metrizable in the sense of Nedev \cite{Ned}).  From now on,
all the topological spaces considered in the paper are
$T_1$-spaces. At first we consider the interplay between the
characters introduced above.

\begin{prop}\label{prop1} Let $X$ be a topological space. Then
\begin{enumerate}
\item $\mathrm{cs}^*_\chi(X)\le\mathrm{cs}_\chi(X)\le \mathrm{sb}_\chi(X)\le\chi(X)$;
\item $\chi(X)=\mathrm{sb}_\chi(X)$ if $X$ is Fr\'echet-Urysohn;
\item $\mathrm{cs}^*_\chi(X)<\aleph_0$ iff $\mathrm{cs}_\chi(X)<\aleph_0$ iff $\mathrm{sb}_\chi(X)<\aleph_0$
iff $\mathrm{cs}^*_\chi(X)=1$ iff $\mathrm{cs}_\chi(X)=1$ iff
$\mathrm{sb}_\chi(X)=1$ iff each convergent sequence in $X$ is
trivial;
\item $\mathrm{sb}_\chi(X)\le 2^{\mathrm{cs}^*_\chi(X)}$;
\item $\mathrm{cs}_\chi(X)\le \mathrm{cs}^*_\chi(X)\cdot\sup\{\big|[\kappa]^{\le\omega}\big|:\kappa<\mathrm{cs}^*_\chi(X)\}\le\big(\mathrm{cs}^*_\chi(X)\big)^{\aleph_0}$
where\newline
$[\kappa]^{\le\omega}=\{A\subset\kappa:|A|\le\aleph_0\}$.
\end{enumerate}
\end{prop}

Here ``iff'' is an abbreviation for ``if and only if''. The Arens'
space $S_2$ and the sequential fan $S_\omega$ give us simple
examples distinguishing between some of the characters considered
above. We recall that the {\em Arens' space} $S_2$ is the set
$\{(0,0),(\frac1n,0),(\frac1n,\frac1{nm}):n,m\in\mathbb{N}\}\subset
\mathbb{R}^2$ carrying the strongest topology inducing the
original planar topology on the convergent sequences
$C_0=\{(0,0),(\frac1n,0):n\in\mathbb{N}\}$ and $C_n=\{(\frac1n,0),
(\frac1n,\frac1{nm}):m\in\mathbb{N}\}$, $n\in\mathbb{N}$. The
quotient space $S_\omega=S_2/C_0$ obtained from the Arens' space
$S_2$ by identifying the points of the sequence $C_0$ is called
the {\em sequential fan}, see \cite{Lin}. The sequential fan
$S_\omega$ is the simplest example of a non-metrizable
Fr\'echet-Urysohn space while $S_2$ is the simplest example of a
sequential space which is not Fr\'echet-Urysohn.

We recall that a topological space $X$ is {\em sequential} if a
subset $A\subset X$ if closed if and only if $A$ is {\em
sequentially closed}\/ in the sense that $A$ contain the limit
point of any sequence $(a_n)\subset A$, convergent in $X$. A
topological space $X$ is {\em Fr\'echet-Urysohn} if for any
cluster point $a\in X$ of a subset $A\subset X$ there is a
sequence $(a_n)\subset A$, convergent to $a$.

Observe that
$\aleph_0=\mathrm{cs}^*_\chi(S_2)=\mathrm{cs}_\chi(S_2)=\mathrm{sb}_\chi(S_2)<\chi(S_2)=\mathfrak
d $ while
$\aleph_0=\mathrm{cs}^*_\chi(S_\omega)=\mathrm{cs}_\chi(S_\omega)<\mathrm{sb}_\chi(S_\omega)=\chi(S_\omega)=\mathfrak
d$. Here $\mathfrak d$ is the well-known in Set Theory small
uncountable cardinal equal to the cofinality of the partially
ordered set $\mathbb{N}^\omega$ endowed with the natural partial
order: $(x_n)\le (y_n)$ iff $x_n\le y_n$ for all $n$, see
\cite{Va}. Besides $\mathfrak d$, we will need two other small
cardinals: $\mathfrak b$ defined as the smallest size of a subset
of uncountable cofinality in $(\mathbb{N}^\omega,\le)$, and
$\mathfrak p$ equal to the smallest size $|\mathcal F|$ of a
family of infinite subsets of $\omega$ closed under finite
intersections and having no infinite pseudo-intersection in the
sense that there is no infinite subset $I\subset\omega$ such that
the complement $I\setminus F$ is finite for any $F\in\mathcal F$,
see \cite{Va}, \cite{vD}. It is known that $\aleph_1\le\mathfrak
p\le\mathfrak b\le\mathfrak d\le\mathfrak c$ where $\mathfrak c$
stands for the size of continuum. Martin Axiom implies $\mathfrak
p=\mathfrak b=\mathfrak d=\mathfrak c$, \cite{MS}. On the other
hand, for any uncountable regular cardinals $\lambda\le\kappa$
there is a model of ZFC with $\mathfrak p=\mathfrak b=\mathfrak
d=\lambda$ and $\mathfrak c=\kappa$, see \cite[5.1]{vD}.

Unlike to the cardinal invariants $\mathrm{cs}_\chi$,
$\mathrm{sb}_\chi$ and $\chi$ which can be distinguished on simple
spaces, the difference between the cardinal invariants
$\mathrm{cs}_\chi$ and $\mathrm{cs}^*_\chi$ is more subtle: they
cannot be distinguished in some models of Set Theory!

\begin{prop}\label{cs=cs*}
Let $X$ be a topological space. Then
$\mathrm{cs}^*_\chi(X)=\mathrm{cs}_\chi(X)$ provided one of the
following conditions is satisfied:
\begin{enumerate}
\item $\mathrm{cs}^*_\chi(X)<\mathfrak p$;
\item $\kappa^{\aleph_0}\le\mathrm{cs}^*_\chi(X)$ for any cardinal
$\kappa<\mathrm{cs}^*_\chi(X)$;
\item $\mathfrak p=\mathfrak c$ and $\lambda^\omega\le\kappa$ for any
cardinals $\lambda<\kappa\ge\mathfrak c$;
\item $\mathfrak p=\mathfrak c$ (this is so under MA)
 and $X$ is countable;
\item the Generalized Continuum Hypothesis holds.
\end{enumerate}
\end{prop}

Unlike to the usual character, the $\mathrm{cs}^*$-,
$\mathrm{cs}$-, and $\mathrm{sb}$-characters behave nicely with
respect to many countable topological operations.

Among such operation there are: the Tychonov product, the
box-product, producing a sequentially homeomorphic copy, taking
image under a sequentially open map, and forming inductive
topologies.

As usual, under {\em the box-product} $\Box_{i\in\mathcal{I}}X_i$
of topological spaces $X_i$, $i\in\mathcal{I}$, we understand the
Cartesian product $\prod_{i\in\mathcal{I}}X_i$ endowed with the
box-product topology generated by the base consisting of products
$\prod_{i\in\mathcal{I}}U_i$ where each $U_i$ is open in $X_i$. In
contrast, by $\prod_{i\in\mathcal{I}}X_i$ we denote the usual
Cartesian product of the spaces $X_i$, endowed with the Tychonov
product topology.

We say that a topological space $X$ carries {\em the inductive
topology} with respect to a cover $\mathcal{C}$ of $X$ if a subset
$F\subset X$ is closed in $X$ if and only if the intersection
$F\cap C$ is closed in $C$ for each element $C\in \mathcal{C}$.
For a cover $\mathcal{C}$ of $X$ let
$\mathrm{ord}(\mathcal{C})=\sup_{x\in
X}\mathrm{ord}(\mathcal{C},x)$ where
$\mathrm{ord}(\mathcal{C},x)=|\{C\in\mathcal{C}:x\in C\}|$. A
topological space $X$ carrying the inductive topology with respect
to a countable cover by closed metrizable (resp. compact, compact
metrizable) subspaces is called an {\em
$\mathcal{M}_\omega$-space} (resp. a {\em $k_\omega$-space}, {\em
$\mathcal{M}\mathcal{K}_\omega$-space}).

A function $f:X\to Y$ between topological spaces is called {\em
sequentially continuous} if for any convergent sequence $(x_n)$ in
$X$ the sequence $(f(x_n))$ is convergent in $Y$ to $f(\lim x_n)$;
$f$ is called a {\em sequential homeomorphism} if $f$ is bijective
and both $f$ and $f^{-1}$ are sequentially continuous. Topological
spaces $X,Y$ are defined to be {\em sequentially homeomorphic} if
there is a sequential homeomorphism $h:X\to Y$. Observe that two
spaces are sequentially homeomorphic if and only if their
sequential coreflexions are homeomorphic. Under {\em the
sequential coreflexion} $\sigma X$ of a topological space $X$ we
understand $X$ endowed with the topology consisting of all
sequentially open subsets of $X$ (a subset $U$ of $X$ is {\em
sequentially open} if its complement is sequentially closed in
$X$; equivalently $U$ is a sequential barrier at each point $x\in
U$). Note that the identity map $\mathrm{id}:\sigma X\to X$ is
continuous while its inverse is sequentially continuous, see
\cite{Lin}.

A map $f:X\to Y$ is {\em sequentially open} if for any point
$x_0\in X$ and a sequence $S\subset Y$ convergent to $f(x_0)$
there is a sequence $T\subset X$ convergent to $x_0$ and such that
$f(T)\subset S$. Observe that a bijective map $f$ is sequentially
open if its inverse $f^{-1}$ is sequentially continuous.

The following technical Proposition is an easy consequence of the
corresponding definitions.

\begin{prop}\label{prop3}
\begin{enumerate}
\item If $X$ is a subspace of a topological space $Y$, then
$\mathrm{cs}^*_\chi(X)\le\mathrm{cs}^*_\chi(Y)$,
$\mathrm{cs}_\chi(X)\le\mathrm{cs}_\chi(Y)$ and
$\mathrm{sb}_\chi(X)\le\mathrm{sb}(Y)$.
\item If $f:X\to Y$ is a surjective continuous sequentially open map between
topological spaces, then $\mathrm{cs}^*_\chi(Y)\le
\mathrm{cs}^*_\chi(X)$ and
$\mathrm{sb}_\chi(Y)\le\mathrm{sb}_\chi(X)$.
\item If $f:X\to Y$ is a surjective sequentially continuous
sequentially open map between topological spaces, then
$\min\{\mathrm{cs}^*_\chi(Y),\aleph_1\}\le
\min\{\mathrm{cs}^*_\chi(X),\aleph_1\}$,
$\min\{\mathrm{cs}_\chi(Y),\aleph_1\}\le
\min\{\mathrm{cs}_\chi(X),\aleph_1\}$, and
$\min\{\mathrm{sb}_\chi(Y),\aleph_1\}\le\min\{\mathrm{sb}_\chi(X),\aleph_1\}$.
\item If $X,Y$ are sequentially homeomorphic topological
spaces, then
$\min\{\mathrm{cs}^*_\chi(X),\aleph_1\}=
\min\{\mathrm{cs}_\chi(X),\aleph_1\}=\min\{\mathrm{cs}_\chi(Y),\aleph_1\}=
\min\{\mathrm{cs}^*_\chi(Y),\aleph_1\}$, and
$\min\{\mathrm{sb}_\chi(Y),\aleph_1\}=\min\{\mathrm{sb}_\chi(X),\aleph_1\}$.
\item $\min\{\mathrm{sb}_\chi(X),\aleph_1\}=\min\{\mathrm{sb}_\chi(\sigma
X),\aleph_1\}\le\mathrm{sb}_\chi(\sigma X)\ge\mathrm{sb}_\chi(X)$
and\newline $\mathrm{cs}_\chi(X)\le\mathrm{cs}_\chi(\sigma X)\ge
\min\{\mathrm{cs}_\chi(\sigma X),\aleph_1\}=
\min\{\mathrm{cs}_\chi(X),\aleph_1\}=\min\{\mathrm{cs}^*_\chi(X),\aleph_1\}=
\min\{\mathrm{cs}^*_\chi(\sigma
X),\aleph_1\}\le\mathrm{cs}^*_\chi(\sigma
X)\ge\mathrm{cs}^*_\chi(X)$ for any topological space $X$.
\item If $X=\prod_{i\in\mathcal{I}}X_i$ is the Tychonov product of topological spaces
$X_i$, $i\in\mathcal{I}$, then $\mathrm{cs}^*_\chi(X)\le
\sum_{i\in\mathcal{I}}\mathrm{cs}^*_\chi(X_i)$,
$\mathrm{cs}_\chi(X)\le
\sum_{i\in\mathcal{I}}\mathrm{cs}_\chi(X_i)$ and
$\mathrm{sb}_\chi(X)\le
\sum_{i\in\mathcal{I}}\mathrm{sb}_\chi(X_i)$.
\item If $X=\Box_{i\in\mathcal{I}}X_i$ is the box-product of topological spaces
$X_i$, $i\in\mathcal{I}$, then $\mathrm{cs}^*_\chi(X)\le
\sum_{i\in\mathcal{I}}\mathrm{cs}^*_\chi(X_i)$ and
$\mathrm{cs}_\chi(X)\le
\sum_{i\in\mathcal{I}}\mathrm{cs}_\chi(X_i)$.
\item If a topological space $X$ carries the inductive topology
with respect to a cover $\mathcal{C}$ of $X$, then
$\mathrm{cs}^*_\chi(X)\le
\mathrm{ord}(\mathcal{C})\cdot\sup_{C\in\mathcal{C}}\mathrm{cs}^*_\chi(C)$.
\item If a topological space $X$ carries the inductive topology
with respect to a point-countable cover $\mathcal{C}$ of $X$, then
$\mathrm{cs}_\chi(X)\le
\sup_{C\in\mathcal{C}}\mathrm{cs}_\chi(C)$.
\item If a topological space $X$ carries the inductive topology
with respect to a point-finite cover $\mathcal{C}$ of $X$, then
$\mathrm{sb}_\chi(X)\le
\sup_{C\in\mathcal{C}}\mathrm{sb}_\chi(C)$.
\end{enumerate}
\end{prop}

Since each first-countable space has countable
$\mathrm{cs}^*$-character, it is natural to consider the class of
topological spaces with countable $\mathrm{cs}^*$-character as a
class of generalized metric spaces. However this class contains
very non-metrizable spaces like $\beta\mathbb{N}$, the Stone-\v
Cech compactification of the discrete space of positive integers.
The reason is that $\beta\mathbb{N}$ contains no non-trivial
convergent sequence. To avoid such pathologies we shall restrict
ourselves by sequential spaces. Observe that a topological space
is sequential provided $X$ carries the inductive topology with
respect to a cover by sequential subspaces. In particular, each
$\mathcal{M}_\omega$-space is sequential and has countable
$\mathrm{cs}^*$-character. Our principal result states that for
topological groups the converse is also true. Under an {\em
$\mathcal{M}_\omega$-group} (resp. {\em
$\mathcal{M}\mathcal{K}_\omega$-group}) we understand a
topological group whose underlying topological space is an
$\mathcal{M}_\omega$-space (resp.
$\mathcal{M}\mathcal{K}_\omega$-space).

\begin{theo}\label{main} Each sequential topological group $G$
with countable $\mathrm{cs}^*$-character is an
$\mathcal{M}_\omega$-group. More precisely, either $G$ is
metrizable or else $G$ contains an open
$\mathcal{M}\mathcal{K}_\omega$-subgroup $H$ and is homeomorphic
to the product $H\times D$ for some discrete space $D$.
\end{theo}

For $\mathcal{M}_\omega$-groups the second part of this theorem
was proven in \cite{Ba1}. Theorem~\ref{main} has many interesting
corollaries.

At first we show that for sequential topological groups with
countable $\mathrm{cs}^*$-character many important cardinal
invariants are countable, coincide or take some fixed values. Let
us remind some definitions, see \cite{En1}. For a topological
space $X$ recall that
\begin{itemize}
\item the {\em pseudocharacter} $\psi(X)$ is the smallest cardinal $\kappa$
such that each one-point set $\{x\}\subset X$ can be written as
the intersection $\{x\}=\cap\mathcal{U}$ of some family
$\mathcal{U}$ of open subsets of $X$ with
$|\mathcal{U}|\le\kappa$;
\item the {\em cellularity} $c(X)$ is the smallest cardinal
$\kappa$ such that $X$ contains no family $\mathcal{U}$ of size
$|\mathcal{U}|>\kappa$ consisting of non-empty pairwise disjoint
open subsets;
\item the {\em Lindel\"of number} $l(X)$ is the smallest cardinal
$\kappa$ such that each open cover of $X$ contains a subcover of
size $\le\kappa$;
\item the {\em density} $d(X)$ is the smallest size of a dense
subset of $X$;
\item the {\em tightness} $t(X)$ is the smallest cardinal $\kappa$
such that for any subset $A\subset X$ and a point $a\in\bar A$
from its closure there is a subset $B\subset A$ of size
$|B|\le\kappa$ with $a\in\bar B$;
\item the {\em extent} $e(X)$ is the smallest cardinal $\kappa$
such that $X$ contains no closed discrete subspace of size
$>\kappa$;
\item the {\em compact covering number} ${kc}(X)$ is the smallest
size of a cover of $X$ by compact subsets;
\item the {\em weight} $w(X)$ is the smallest size of a base of
the topology of $X$;
\item the {\em network weight} $nw(X)$ is the smallest size
$|\mathcal N|$ of a topological network for $X$ (a family
$\mathcal N$ of subsets of $X$ is a {\em topological network} if
for any open set $U\subset X$ and any point $x\in U$ there is
$N\in\mathcal N$ with $x\in N\subset U$);
\item the {\em $k$-network weight}  $knw(X)$ is the smallest size
$|\mathcal N|$ of a $k$-network for $X$ (a family $\mathcal N$ of
subsets of $X$ is a {\em $k$-network} if for any open set
$U\subset X$ and any compact subset $K\subset U$ there is a finite
subfamily $\mathcal M\subset\mathcal N$ with $K\subset\cup\mathcal
M\subset U$).
\end{itemize}

For each topological space $X$ these cardinal invariants relate as
follows: $$\max\{c(X),l(X),e(X)\}\le nw(X)\le knw(X)\le w(X).$$
For metrizable spaces all of them are equal, see
\cite[4.1.15]{En1}.

In the class of $k$-spaces there is another cardinal invariant,
the $k$-ness introduced by E.~van Douwen, see \cite[\S8]{vD}. We
remind that a topological space $X$ is called a {\em $k$-space} if
it carries the inductive topology with respect to the cover of $X$
by all compact subsets. It is clear that each sequential space is
a $k$-space. The {\em $k$-ness} $k(X)$ of a $k$-space is the
smallest size $|\mathcal{K}|$ of a cover $\mathcal{K}$ of $X$ by
compact subsets such that $X$ carries the inductive topology with
respect to the cover $\mathcal{K}$. It is interesting to notice
that $k(\mathbb{N}^\omega)=\mathfrak d$ while $k(\mathbb
Q)=\mathfrak b$, see \cite{vD}. Proposition~\ref{prop3}(8) implies
that $\mathrm{cs}_\chi^*(X)\le k(X)\cdot\psi(X)\ge kc(X)$ for each
$k$-space $X$. Observe also that a topological space $X$ is a
$k_\omega$-space if and only if $X$ is a $k$-space with
$k(X)\le\aleph_0$.

Besides  cardinal invariants we shall consider an ordinal
invariant, called the sequential order. Under {\em the sequential
closure} $A^{(1)}$ of a subset $A$ of a topological space $X$ we
understand the set of all limit point of sequences $(a_n)\subset
A$, convergent in $X$. Given an ordinal $\alpha$ define the
$\alpha$-th sequential closure $A^{(\alpha)}$ of $A$ by
transfinite induction:
$A^{(\alpha)}=\bigcup_{\beta<\alpha}(A^{(\beta)})^{(1)}$. Under
the {\em sequential order} $\mathrm{so}(X)$ of a topological space
$X$ we understand the smallest ordinal $\alpha$ such that
$A^{(\alpha+1)}=A^{(\alpha)}$ for any subset $A\subset X$. Observe
that a topological space $X$ is Fr\'echet-Urysohn if and only if
$\mathrm{so}(X)\le 1$; $X$ is sequential if and only if
$\mathrm{cl}_X(A)=A^{(\mathrm{so}(X))}$ for any subset $A\subset
X$.

 Besides purely topological invariants we shall also consider
a cardinal invariant, specific for topological groups. For a
topological group $G$ let $ib(G)$, the {\em boundedness index} of
$G$ be the smallest cardinal $\kappa$ such that for any nonempty
open set $U\subset G$ there is a subset $F\subset G$ of size
$|F|\le\kappa$ such that $G=F\cdot U$. It is known that
${ib}(G)\le\min\{c(G),l(G),e(G)\}$ and $w(G)=ib(G)\cdot\chi(G)$
for each topological group, see \cite{Tk}.

\begin{theo}\label{cardinal} Each sequential
topological group $G$ with countable $\mathrm{cs}^*$-character has
the following properties: $\psi(G)\le\aleph_0$,
$\mathrm{sb}_\chi(G)=\chi(G)\in\{1,\aleph_0,\mathfrak d\}$,
${ib}(G)=c(G)=d(G)=l(G)=e(G)={nw}(G)={knw}(G)$, and
$\mathrm{so}(G)\in\{1,\omega_1\}$.
\end{theo}

We shall derive from Theorems~\ref{main} and \ref{cardinal} an
unexpected metrization theorem for topological groups. But first
we need to remind the definitions of some of $\alpha_i$-spaces,
$i=1,\dots,6$ introduced by A.V.~Arkhangelski in \cite{Ar2},
\cite{Ar3}. We also define a wider class of $\alpha_7$-spaces.

A topological space $X$ is called
\begin{itemize}
\item an {\em $\alpha_1$-space} if for any sequences
$S_n\subset X$, $n\in\omega$, convergent to a point $x\in X$ there
is a sequence $S\subset X$ convergent to $x$ and such that
$S_n\setminus S$ is finite for all $n$;
\item an {\em $\alpha_4$-space} if for any sequences
$S_n\subset X$, $n\in\omega$, convergent to a point $x\in X$ there
is a sequence $S\subset X$ convergent to $x$ and such that
$S_n\cap S\ne\emptyset$ for infinitely many sequences $S_n$;
\item an {\em $\alpha_7$-space} if for any sequences
$S_n\subset X$, $n\in\omega$, convergent to a point $x\in X$ there
is a sequence $S\subset X$ convergent to some point $y$ of $X$ and
such that $S_n\cap S\ne\emptyset$ for infinitely many sequences
$S_n$;
\end{itemize}

Under a sequence converging to a point $x$ of a topological space
$X$ we understand any countable infinite subset $S$ of $X$ such
that $S\setminus U$ if finite for any neighborhood $U$ of $x$.
Each $\alpha_1$-space is $\alpha_4$ and each $\alpha_4$-space is
$\alpha_7$. Quite often $\alpha_7$-spaces are $\alpha_4$, see
Lemma~\ref{alpha4}. Observe also that each sequentially compact
space is $\alpha_7$. It can be shown that a topological space $X$
is an $\alpha_7$-space if and only if it contains no closed copy
of the sequential fan $S_\omega$ in its sequential coreflexion
$\sigma X$.  If $X$ is an $\alpha_4$-space, then $\sigma X$
contains no topological copy of $S_\omega$.

We remind that a topological group $G$ is {\em Weil complete} if
it is complete in its left (equivalently, right) uniformity.
According to \cite[4.1.6]{PZ}, each $k_\omega$-group is Weil
complete. The following metrization theorem can be easily derived
from Theorems~\ref{main}, \ref{cardinal} and elementary properties
of $\mathcal{MK}_\omega$-groups.

\begin{theo}\label{metr} A sequential topological group $G$ with countable
$\mathrm{cs}^*$-character is metrizable if one of the following
conditions is satisfied:
\begin{enumerate}
\item $\mathrm{so}(G)<\omega_1$;
\item $\mathrm{sb}_\chi(G)<\mathfrak d$;
\item ${ib}(G)< k(G)$;
\item $G$ is Fr\'echet-Urysohn;
\item $G$ is an $\alpha_7$-space;
\item $G$ contains no closed copy of $S_\omega$ or $S_2$;
\item $G$ is not Weil complete;
\item $G$ is Baire;
\item ${ib}(G)<|G|<2^{\aleph_0}$.
\end{enumerate}
\end{theo}

According to Theorem~\ref{main}, each sequential topological group
with countable $\mathrm{cs}^*$-character is an
$\mathcal{M}_\omega$-group. The first author has proved in
\cite{Ba2} that the topological structure of a non-metrizable
punctiform $\mathcal{M}_\omega$-group is completely determined by
its density and the compact scatteredness rank.

Recall that a topological space $X$ is {\em punctiform} if $X$
contains no compact connected subspace containing more than one
point, see \cite[1.4.3]{En2}. In particular, each zero-dimensional
space is punctiform.

Next, we remind the definition of the scatteredness height. Given
a topological space $X$  let $X_{(1)}\subset X$ denote the set of
all non-isolated points of $X$. For each ordinal $\alpha$ define
the $\alpha$-th derived set $X_{(\alpha)}$ of $X$ by transfinite
induction:
$X_{(\alpha)}=\bigcap_{\beta<\alpha}(X_{(\beta)})_{(1)}$. Under
the {\em scatteredness height} $\mathrm{sch}(X)$ of $X$ we
understand the smallest ordinal $\alpha$ such that
$X_{(\alpha+1)}=X_{(\alpha)}$. A topological space $X$ is {\em
scattered} if $X_{(\alpha)}=\emptyset$ for some ordinal $\alpha$.
Under the {\em compact scatteredness rank}\/  of a topological
space $X$ we understand the ordinal
$\mathrm{scr}(X)=\sup\{\mathrm{sch}(K): K$ is a scattered compact
subspace of $X\}$.

\begin{theo}\label{zero} Two non-metrizable sequential punctiform
topological groups $G,H$ with countable $\mathrm{cs}^*$-character
are homeomorphic if and only if $d(G)=d(H)$ and
$\mathrm{scr}(G)=\mathrm{scr}(H)$.
\end{theo}

This theorem follows from Theorem~\ref{main} and ``Main  Theorem"
of \cite{Ba2} asserting that two non-metrizable punctiform
$\mathcal{M}_\omega$-groups $G$, $H$ are homeomorphic if and only
if $d(G)=d(H)$ and $\mathrm{scr}(G)=\mathrm{scr}(H)$. For
countable $k_\omega$-groups this fact was proven by E.Zelenyuk
\cite{Ze}.

The topological classification of non-metrizable sequential
locally convex spaces with countable $\mathrm{cs}^*$-character is
even more simple. Any such a space is homeomorphic either to
$\mathbb{R}^\infty$ or to $\mathbb{R}^\infty\times Q$ where
$Q=[0,1]^\omega$ is the Hilbert cube and $\mathbb{R}^\infty$ is a
linear space of countable algebraic dimension, carrying the
strongest locally convex topology. It is well-known that this
topology is inductive with respect to the cover of
$\mathbb{R}^\infty$ by finite-dimensional linear subspaces. The
topological characterization of the spaces $\mathbb{R}^\infty$ and
$\mathbb{R}^\infty\times Q$ was given in \cite{Sa}. In \cite{Ba3}
it was shown that each infinite-dimensional locally convex
$\mathcal{M}\mathcal{K}_\omega$-space is homeomorphic to
$\mathbb{R}^\infty$ or $\mathbb{R}^\infty\times Q$. This result
together with Theorem~\ref{main} implies the following
classification

\begin{cor} Each non-metrizable sequential locally convex
space with countable $\mathrm{cs}^*$-cha\-rac\-ter is homeomorphic
to $\mathbb{R}^\infty$ or $\mathbb{R}^\infty\times Q$.
\end{cor}

As we saw in Theorem~\ref{cardinal}, each sequential topological
group with countable $\mathrm{cs}^*$-character has countable
pseudocharacter. The proof of this result is based on the fact
that compact subsets of sequential topological groups with
countable $\mathrm{cs}^*$-character are first countable. This
naturally leads to a conjecture that compact spaces  with
countable $\mathrm{cs}^*$-character are first countable.
Surprisingly, but this conjecture is false: assuming the Continuum
Hypothesis N.~Yakovlev \cite{Ya} has constructed a scattered
sequential compactum which has countable $\mathrm{sb}$-character
but fails to be first countable. In \cite{Nyi2} P.Nyikos pointed
out that the Yakovlev construction still can be carried under the
assumption $\mathfrak b=\mathfrak c$. More precisely, we have

\begin{prop}\label{yakovlev} Under $\mathfrak
b=\mathfrak c$ there is a regular locally compact locally
countable space $Y$ whose one-point compactification $\alpha Y$ is
sequential and satisfies $\aleph_0=\mathrm{sb}_\chi(\alpha
Y)<\psi(\alpha Y)=\mathfrak c$.
\end{prop}

We shall use the above proposition to construct examples of
topologically homogeneous spaces with countable
$\mathrm{cs}$-character and uncountable pseudocharacter. This
shows that Theorem~\ref{cardinal} is specific for topological
groups and cannot be generalized to topologically homogeneous
spaces. We remind that a topological space $X$ is {\em
topologically homogeneous} if for any points $x,y\in X$ there is a
homeomorphism $h:X\to X$ with $h(x)=y$.

\begin{theo}\label{dowker}\hfill
\begin{enumerate}
\item There is a topologically homogeneous countable regular
$k_\omega$-space $X_1$ with
$\aleph_0=\mathrm{sb}_\chi(X_1)<\chi(X_1)=\mathfrak d$ and
$\mathrm{so}(X_1)=\omega$;
\item Under $\mathfrak b=\mathfrak c$ there is a sequential
topologically homogeneous zero-dimensional $k_\omega$-space $X_2$
with $\aleph_0=\mathrm{cs}_\chi(X_2)<\psi(X_2)=\mathfrak c$;
\item Under $\mathfrak b=\mathfrak c$ there is a sequential topologically homogeneous
 totally disconnected space $X_3$ with
$\aleph_0=\mathrm{sb}_\chi(X_3)<\psi(X_3)=\mathfrak c$.
\end{enumerate}
\end{theo}

We remind that a space $X$ is {\em totally disconnected} if for
any distinct points $x,y\in X$ there is a continuous function
$f:X\to\{0,1\}$ such that $f(x)\ne f(y)$, see \cite{En2}.

\begin{rem} The space $X_1$ from Theorem~\ref{dowker}(1) is the well-known
Arkhangelski-Franklin example \cite{AF} (see also \cite[10.1]{Co})
of a countable topologically homogeneous $k_\omega$-space,
homeomorphic to no topological group (this also follows from
Theorem~\ref{cardinal}). On the other hand, according to
\cite{Ze2}, each topologically homogeneous countable regular space
(in particular, $X_1$) is homeomorphic to a {\em quasitopological
group}, that is a topological space endowed with a separately
continuous group operation with continuous inversion. This shows
that Theorem~\ref{cardinal} cannot be generalized onto
quasitopological groups (see however \cite{Zd} for generalizations
of Theorems~\ref{main} and \ref{cardinal} to some other
topologo-algebraic structures).
\end{rem}

Next, we find conditions under which a space with countable
$\mathrm{cs}^*$-character is first-countable or has countable
$\mathrm{sb}$-character. Following \cite{Ar4} we define a
topological space $X$ to be {\em $c$-sequential\/} if for each
closed subspace $Y\subset X$ and each non-isolated point $y$ of
$Y$ there is a sequence $(y_n)\subset Y\setminus\{y\}$ convergent
to $y$. It is clear that each sequential space is $c$-sequential.
A point $x$ of a topological space $X$ is called {\em regular
$G_\delta$} if $\{x\}=\cap\mathcal B$ for some countable family
$\mathcal B$ of closed neighborhood of $x$ in $X$, see \cite{Lin}.

First we characterize spaces with countable
$\mathrm{sb}$-character (the first three items of this
characterization were proved by Lin \cite[3.13]{Lin} in terms of
(universally) $csf$-countable spaces).

\begin{prop}\label{sb} For a Hausdorff space $X$ the following conditions are
equivalent:
\begin{enumerate}
\item $X$ has countable $\mathrm{sb}$-character;
\item $X$ is an $\alpha_1$-space with countable $\mathrm{cs}^*$-character;
\item $X$ is an $\alpha_4$-space with countable $\mathrm{cs}^*$-character;
\item $\mathrm{cs}^*_\chi(X)\le\aleph_0$ and $\mathrm{sb}_\chi(X)<\mathfrak p$.

\hskip-50pt Moreover, if $X$ is $c$-sequential and each point of
$X$ is regular $G_\delta$, then the conditions (1)--(4) are
equivalent to:

\item $\mathrm{cs}^*_\chi(X)\le\aleph_0$ and $\mathrm{sb}_\chi(X)<\mathfrak d$.
\end{enumerate}
\end{prop}

Next, we give a characterization of first-countable spaces in the
same spirit (the equivalences
$(1)\Leftrightarrow(2)\Leftrightarrow(5)$ were proved by Lin
\cite[2.8]{Lin}).

\begin{prop}\label{chi} For a Hausdorff space $X$ with countable
$\mathrm{cs}^*$-character the following conditions are equivalent:
\begin{enumerate}
\item $X$ is first-countable;
\item $X$ is Fr\'echet-Urysohn and has countable $\mathrm{sb}$-character;
\item $X$ is Fr\'echet-Urysohn $\alpha_7$-space;
\item $\chi(X)<\mathfrak p$ and $X$ has countable tightness.

\hskip-50pt Moreover, if each point of $X$ is regular $G_\delta$,
then the conditions (1)--(4) are equivalent to:

\item $X$ is a sequential space containing no closed copy of $S_2$
or $S_\omega$;
\item $X$ is a sequential space with $\chi(X)<\mathfrak d$.
\end{enumerate}
\end{prop}

For Fr\'echet-Urysohn (resp. dyadic) compacta the countability of
the $\mathrm{cs}^*$-character is equivalent to the first
countability (resp. the metrizability). We remind that a compact
Hausdorff space $X$ is called {\em dyadic} if $X$ is a continuous
image of the Cantor discontinuum $\{0,1\}^\kappa$ for some
cardinal $\kappa$.

\begin{prop}\label{dyadic}\hfill
\begin{enumerate}
\item A Fr\'echet-Urysohn countably compact space is first-countable if and only if it has countable
$\mathrm{cs}^*$-character.
\item  A dyadic compactum is metrizable
if and only if its has countable $\mathrm{cs}^*$-character.
\end{enumerate}
\end{prop}

In light of Proposition~\ref{dyadic}(1) one can suggest that
$\mathrm{cs}^*_\chi(X)=\chi(X)$ for any compact Fr\'echet-Urysohn
space $X$. However that is not true: under CH,
$\mathrm{cs}_\chi(\alpha D)\ne\chi(\alpha D)$ for the one-point
compactification $\alpha D$ of a discrete space $D$ of size
$|D|=\aleph_2$. Surprisingly, but the problem of calculating the
$\mathrm{cs}^*$- and $\mathrm{cs}$-characters of the spaces
$\alpha D$ is not trivial and the definitive answer is known only
under the Generalized Continuum Hypothesis. First we note that the
cardinals $\mathrm{cs}^*_\chi(\alpha D)$ and
$\mathrm{cs}_\chi(\alpha D)$ admit an interesting interpretation
which will be used for their calculation.

\begin{prop}\label{alphaD} Let $D$ be an infinite discrete space.
Then
\begin{enumerate}
\item $\mathrm{cs}^*_\chi(\alpha D)=\min\{w(X): X$ is a (regular zero-dimensional)
topological space of size $|X|=|D|$ containing non no-trivial
convergent sequence$\}$;
\item $\mathrm{cs}_\chi(\alpha D)=\min\{w(X): X$ is a (regular zero-dimensional)
topological space of size $|X|=|D|$ containing no countable
non-discrete subspace$\}$.
\end{enumerate}
\end{prop}

 For
a cardinal $\kappa$ we put $\log \kappa=\min\{\lambda:\kappa\le
2^\lambda\}$ and $\mathrm{cof}([\kappa]^{\le\omega})$ be the
smallest size of a collection $\mathcal
C\subset[\kappa]^{\le\omega}$ such that each at most countable
subset $S\subset\kappa$ lies in some element $C\in\mathcal C$.
Observe that $\mathrm{cof}([\kappa]^{\le\omega})\le\kappa^\omega$
but sometimes the inequality can be strict:
$1=\mathrm{cof}([\aleph_0]^{\le\omega})<\aleph_0$ and
$\aleph_1=\mathrm{cof}([\aleph_1]^{\le\omega})<\aleph_1^{\aleph_0}$.
In the following proposition we collect all the information on the
cardinals $\mathrm{cs}^*_\chi(\alpha D)$ and
$\mathrm{cs}_\chi(\alpha D)$ we know.

\begin{prop}\label{alphaD1} Let $D$ be an uncountable discrete
space. Then
\begin{enumerate}
\item $\aleph_1\cdot\log|D|\le
\mathrm{cs}^*_\chi(\alpha D)\le\mathrm{cs}_\chi(\alpha
D)\le\min\{|D|,
2^{\aleph_0}\cdot\mathrm{cof}([\log|D|]^{\le\omega})\}$
while\newline $\mathrm{sb}_\chi(\alpha D)=\chi(\alpha D)=|D|$;
\item $\mathrm{cs}^*_\chi(\alpha D)=\mathrm{cs}_\chi(\alpha
D)=\aleph_1\cdot\log|D|$ under GCH.
\end{enumerate}
\end{prop}

In spite of numerous efforts some annoying problems concerning
$\mathrm{cs}^*$- and $\mathrm{cs}$-characters still rest open.

\begin{prb} Is there a (necessarily consistent) example of a
space $X$ with $\mathrm{cs}^*_\chi(X)\ne\mathrm{cs}_\chi(X)$? In
particular, is $\mathrm{cs}^*_\chi(\alpha \mathfrak
c)\ne\mathrm{cs}_\chi(\alpha\mathfrak c)$ in some model of ZFC?
\end{prb}

In light of Proposition~\ref{alphaD} it is natural to consider the
following three cardinal characteristics of the continuum which
seem to be new:
\begin{align*}
\mathfrak w_1=\min\{w(X): &\; X\mbox{ is a topological space of
size $|X|=\mathfrak c$ containing no}\\ &\mbox{\hskip10pt
non-trivial convergent sequence}\};\\ \mathfrak w_2=\min\{w(X):&\;
X\mbox{ is a topological space of size $|X|=\mathfrak c$
containing no}\\ &\mbox{\hskip10pt non-discrete countable
subspace}\};\\ \mathfrak w_3=\min\{w(X):&\; X\mbox{ is a $P$-space
of size $|X|=\mathfrak c$}\}.
\end{align*}

As expected, a {\em $P$-space} is a $T_1$-space whose any
$G_\delta$-subset is open. Observe that $\mathfrak
w_1=\mathrm{cs}^*_\chi(\alpha\mathfrak c)$ while $\mathfrak
w_2=\mathrm{cs}_\chi(\alpha \mathfrak c)$. It is clear that
$\aleph_1\le \mathfrak w_1\le\mathfrak w_2\le\mathfrak
w_3\le\mathfrak c$ and hence  the cardinals $\mathfrak w_i$,
$i=1,2,3$, fall into the category of small uncountable cardinals,
see \cite{Va}.

\begin{prb} Are the cardinals $\mathfrak w_i$, $i=1,2,3$, equal
to (or can be estimated via) some known small uncountable
cardinals considered in Set Theory? Is $\mathfrak w_1<\mathfrak
w_2<\mathfrak w_3$ in some model of ZFC?
\end{prb}

Our next question concerns the assumption $\mathfrak b=\mathfrak
c$ in Theorem~\ref{dowker}.

\begin{prb} Is there a ZFC-example of a sequential space $X$ with
$\mathrm{sb}_\chi(X)<\psi(X)$ or at least
$\mathrm{cs}^*_\chi(X)<\psi(X)$?
\end{prb}

Propositions~\ref{prop1} and \ref{sb} imply that
$\mathrm{sb}_\chi(X)\in\{1,\aleph_0\}\cup[\mathfrak d,\mathfrak
c]$ for any $c$-sequential topological space $X$ with countable
$\mathrm{cs}^*$-character. On the other hand, for a sequential
topological group $G$ with countable $\mathrm{cs}^*$-character we
have a more precise estimate
$\mathrm{sb}_\chi(G)\in\{1,\aleph_0,\mathfrak d\}$.

\begin{prb} Is $\mathrm{sb}_\chi(X)\in\{1,\aleph_0,\mathfrak d\}$ for
any sequential space $X$ with countable $\mathrm{cs}^*$-character?
\end{prb}

As we saw in Proposition~\ref{dyadic}, $\chi(X)\le\aleph_0$ for
any Fr\'echet-Urysohn compactum $X$ with
$\mathrm{cs}_\chi(X)\le\aleph_0$.

\begin{prb} Is $\mathrm{sb}_\chi(X)\le\aleph_0$ for any sequential
(scattered) compactum $X$ with $\mathrm{cs}_\chi(X)\le\aleph_0$?
\end{prb}

Now we pass to proofs of our results.

\section*{On sequence trees in topological groups}

Our basic instrument in proofs of main results is the concept of a
sequence tree. As usual, under a {\em tree} we understand a
partially ordered subset $(T,\le)$ such that for each $t\in T$ the
set $\downarrow t=\{\tau\in T:\sigma\le t\}$ is well-ordered by
the order $\le$. Given an element $t\in T$ let $\uparrow
t=\{\tau\in T:\tau\ge t\}$ and $\mathrm{succ}(t)=\min(\uparrow
t\setminus\{t\})$ be the set of successors of $t$ in $T$. A
maximal linearly ordered subset of a tree $T$ is called a {\em
branch} of $T$. By $\max T$ we denote the set of maximal elements
of the tree $T$.

\begin{dfn} Under a {\em sequence tree} in a topological space $X$ we understand a tree
$(T,\le)$ such that
\begin{itemize}
\item $T\subset X$;
\item $T$ has no infinite branch;
\item for each $t\notin\max T$ the set $\min(\uparrow t\setminus\{t\})$
of successors of $t$ is countable and converges to $t$.
\end{itemize}
\end{dfn}

Saying that a subset $S$ of a topological space $X$ converges to a
point $t\in X$ we mean that for each neighborhood $U\subset X$ of
$t$ the set $S\setminus U$ is finite.

The following lemma is well-known and can be easily proven by
transfinite induction (on the ordinal $s(a,A)=\min\{\alpha:a\in
A^{(\alpha)}\}$ for a subset $A$ of a sequential space and a point
$a\in\bar A$ from its closure)

\begin{lem}\label{l1} A point $a\in X$ of a sequential topological space
$X$ belongs to the closure of a subset $A\subset X$ if and only if
there is a sequence tree $T\subset X$ with $\min T=\{a\}$ and
$\max T\subset A$.
\end{lem}

For subsets $A,B$ of a group $G$ let $A^{-1}=\{x^{-1}:x\in
A\}\subset G$ be the inversion of $A$ in $G$ and $AB=\{xy:x\in
A,\; y\in B\}\subset G$ be the product of $A,B$ in $G$. The
following two lemmas will be used in the proof of
Theorem~\ref{main}.

\begin{lem}\label{l2} A sequential
subspace $F\subset X$ of a topological group $G$ is first
countable if the subspace $F^{-1}F\subset G$ has countable
$\mathrm{sb}$-character at the unit $e$ of the group $G$.
\end{lem}

\begin{proof}
Our proof starts with the observation that it is sufficient to
consider the case $ e\in F $ and prove that $ F $ has countable
character at $ e $.

Let $\{S_n :n\in\omega\}$ be a decreasing $\mathrm{sb}$-network at
$ e $ in $ F^{-1} F $. First we show that for every $ n\in\omega $
there exists $ m>n $ such that $ S_{m}^{2}\cap (F^{-1}F)\subset
S_n $. Otherwise, for every $ m\in\omega $ there would exist $ x_m
,y_m \in S_m $ with $x_m y_m\in (F^{-1}F)\setminus S_n$. Taking
into account that $\lim_{m\rightarrow \infty } x_m =
\lim_{m\rightarrow\infty } y_m =e $, we get
$\lim_{m\rightarrow\infty}x_m y_m =e $. Since $S_n$ is a
sequential barrier at $e$, there is a number $m$ with $x_my_m\in
S_n$, which contradicts to the choice of the points $x_m,y_m$.

Now let us show that for all $ n\in\omega $ the set $ S_n \cap F $
is a neighborhood of $ e $ in $ F $. Suppose, conversely, that
$e\in\mathrm{cl}_F(F\setminus S_{n_0}) $ for some $ n_0 \in\omega
$.

By Lemma \ref{l1} there exists a sequence tree $T\subset F $,
$\min T=\{e\}$ and $\max T\subset F\setminus S_{n_0} $. To get a
contradiction we shall construct an infinite branch of $ T $. Put
$ x_0 =e $ and let $ m_0 $ be the smallest integer such that $
S_{m_0}^2 \cap F^{-1}F\subset S_{n_0}$.

By induction, for every $i\ge 1$ find a number $m_i>m_{i-1}$ with
$S_{m_i}^2\cap F^{-1}F\subset S_{m_{i-1}}$ and a point $x_i\in
\mathrm{succ}(x_{i-1})\cap (x_{i-1}S_{m_i})$. To show that such a
choice is always possible, it suffices to verify that
$x_{i-1}\notin \max T$. It follows from the inductive construction
that $x_{i-1}\in F\cap (S_{m_0}\cdots S_{m_{i-1}})\subset F\cap
S_{m_0}^2\subset S_{n_0}$ and thus $x_{i-1}\notin \max T$ because
$\max T\subset F\setminus S_{n_0}$.

Therefore we have constructed an infinite branch
$\{x_i:i\in\omega\}$ of the sequence tree $T$ which is not
possible. This contradiction finishes the proof.
\end{proof}

\begin{lem}\label{l3} A sequential $\alpha_7$-subspace $F$ of a topological
group $G$ has countable $\mathrm{sb}$-character provided the
subspace $F^{-1}F\subset G$ has countable $\mathrm{cs}$-character
at the unit $e$ of $G$.
\end{lem}

\begin{proof}
Suppose that $ F\subset G $ is a sequential $\alpha_7$-space with
$\mathrm{cs}_\chi(F^{-1}F,e)\le\aleph_0$.  We have to prove that
$\mathrm{sb}_\chi(F,x)\le\aleph_0$ for any point $x\in F$.
 Replacing $F$ by $Fx^{-1}$, if necessary, we can assume that
 $ x=e $ is the unit of the group $G$.
 Fix a countable family $\mathcal{A}$ of subsets of $G$
closed under group products in $G$, finite unions and finite
intersections, and such that $F^{-1}F\in\mathcal A$ and
$\mathcal{A}|F^{-1}F=\{A\cap (F^{-1}F):A\in\mathcal{A}\}$ is a
$\mathrm{cs}$-network at $ e $ in $ F^{-1}F$.
We claim that the collection $\mathcal{A}|F=\{A\cap F:A\in\mathcal
A\}$ is a $\mathrm{sb}$-network at $e$ in $F$.

Assuming the converse, we would find an open neighborhood $
U\subset G $ of $ e $ such that for any element $ A\in\mathcal{A}$
with $ A\cap F\subset U $ the set $A\cap F $ fails to be a
sequential barrier at $ e $ in $ F $.

Let $\mathcal{A}'=\{A\in\mathcal{A}:A\subset F\cap U\}=\{A_n
:n\in\omega\}$ and $ B_n =\bigcup_{k\le n}A_{k}$. Let $m_{-1}=0$
and $U_{-1}\subset U$ be any closed neighborhood of $e$ in $G$. By
induction, for every $k\in\omega$ find a number $m_k>m_{k-1}$, a
closed neighborhood $U_k\subset U_{k-1}$ of $e$ in $G$, and a
sequence $(x_{k,i})_{i\in\omega}$ convergent to $e$ so that the
following conditions are satisfied:
\smallskip

(i) $\{x_{k,i}:i\in\omega\}\subset U_{k-1}\cap F\setminus
B_{m_{k-1}}$;

(ii) the set $F_k=\{x_{n,i}: n\le k,\; i\in\omega\}\setminus
B_{m_k}$ is finite;

(iii) $U_k\cap(F_k\cup\{x_{i,j}:i,j\le k\})=\emptyset$ and
$U_k^2\subset U_{k-1}$.
\smallskip

The last condition implies that $U_0U_1\cdots U_k\subset U$ for
every $k\ge 0$.

Consider the subspace $X=\{x_{k,i}:k,i\in\omega\}$ of $F$ and
observe that it is discrete (in itself). Denote by $\bar X$ the
closure of $X$ in $F$ and observe that $\bar X\setminus X$ is
closed in $F$. We claim that $e$ is an isolated point of $\bar
X\setminus X$. Assuming the converse and applying Lemma~\ref{l1}
we would find a sequence tree $T\subset \bar X$ such that $\min
T=\{e\}$, $\max T\subset X$, and $\mathrm{succ}(e)\subset \bar
X\setminus X$.

By induction, construct a (finite) branch $(t_i)_{i\le n+1}$ of
the tree $T$ and a sequence $\{C_i:i\le n\}$ of elements of the
family $\mathcal A$ such that $t_0=e$,
$|\mathrm{succ}(t_i)\setminus t_iC_i|<\aleph_0$ and $C_i\subset
U_i\cap (F^{-1}F)$, $t_{i+1}\in\mathrm{succ}(t_i)\cap t_iC_i$, for
each $i\le n$. Note that the infinite set
$\sigma=\mathrm{succ}(t_n)\cap t_nC_n\subset X$ converges to the
point $t_n\ne e$.

On the other hand, $\sigma\subset t_nC_n\subset
t_{n-1}C_{n-1}C_n\subset\dots\subset t_0C_0\cdots C_n\subset
U_0\cdots U_n\subset U$. It follows from our assumption on
$\mathcal A$ that $C_0\cdots C_n\in \mathcal A$ and thus
$(C_0\cdots C_n)\cap F\subset B_{m_k}$ for some $k$. Consequently,
$\sigma\subset X\cap B_{m_k}$ and $\sigma\subset\{x_{j,i}:j\le
k,\;i\in\omega\}$ by the item (i) of the construction of $X$.
Since $e$ is a unique cluster point of the set $\{x_{j,i}:j\le
k,\; i\in\omega\}$, the sequence $\sigma$ cannot converge to
$t_n\ne e$, which is a contradiction.

Thus $e$ is an isolated point of $\bar X\setminus X$ and
consequently, there is a closed neighborhood $W$ of $e$ in $G$
such that the set $V=(\{e\}\cup X)\cap W$ is closed in $F$.

For every $n\in\omega$ consider the sequence
$S_n=W\cap\{x_{n,i}:i\in\omega\}$ convergent to $e$. Since $F$ is
an $\alpha_7$-space, there is a convergent sequence $S\subset F$
such that $S\cap S_n\ne\emptyset$ for infinitely many sequences
$S_n$. Taking into account that $V$ is a closed subspace of $F$
with $|V\cap S|=\aleph_0$, we conclude that the limit point $\lim
S$ of $S$ belongs to the set $V$. Moreover, we can assume that
$S\subset V$. Since the space $X$ is discrete, $\lim S\in
V\setminus X=\{e\}$. Thus the sequence $S$ converges to $e$. Since
$\mathcal A'$ is a $\mathrm{cs}$-network at $e$ in $F$, there is a
number $n\in\omega$ such that $A_n$ contains almost all members of
the sequence $S$. Since $S_m\cap (S_k\cup A_n)=\emptyset$ for $m>
k\ge n$, the sequence $S$ cannot meet infinitely many sequences
$S_m$. But this contradicts to the choice of $S$.
\end{proof}

Following \cite[\S8]{vD} by $\mathbb L$ we denote the countable
subspace of the plane $\mathbb{R}^2$: $$\mathbb
L=\{(0,0),(\tfrac1n,\tfrac1{nm}):n,m\in\mathbb{N}\}\subset\mathbb{R}^2.$$
The space $\mathbb L$ is locally compact at each point except for
$(0,0)$. Moreover, according to Lemma 8.3 of \cite{vD}, a first
countable space $X$ contains a closed topological copy of the
space $\mathbb L$ if and only if $X$ is not locally compact.

The following important lemma was proven in \cite{Ba1} for normal
sequential groups.

\begin{lem}\label{l4} If a sequential topological group $G$ contains a closed copy of
the space $\mathbb L$, then $G$ is an $\alpha_7$-space.
\end{lem}

\begin{proof} Let $h:\mathbb L\to G$ be a closed embedding
and let $x_0=h(0,0)$, $x_{n,m}=h(\frac1n,\frac1{nm})$ for
$n,m\in{\mathbb N}$. To show that $G$ is an $\alpha_7$-space, for
every $n\in{\mathbb N}$ fix a sequence $(y_{n,m})_{m\in{\mathbb
N}}\subset G$, convergent to the unit $e$ of $G$. Denote by
$\ast:G\times G\to G$ the group operation on $G$.

It is easy to verify that for every $n$ the subspace
$D_n=\{x_{n,m}\ast y_{n,m}:m\in\mathbb{N} \}$ is closed and
discrete in $G$.  Hence there exists $ k_{n}\in\mathbb{N} $ such
that $x_0\not=x_{n,m}\ast y_{n,m}$ for all $ m>k_{n}$. Consider
the subset $$A=\{x_{n,m}\ast y_{n,m}:n>0,\; m>k_n\}$$ and using
the continuity of the group operation, show that $ x_{0}\not\in A
$ is a cluster point of $A$ in $G$. Consequently, the set $A$ is
not closed and by the sequentiality of $G$, there is a sequence
$S\subset A$ convergent to a point $a\notin A$. Since every space
$D_n$ is closed and discrete in $G$, we may replace $S$ by a
subsequence, and assume that $|S\cap D_n|\le 1$ for every $n\in
{\mathbb N}$. Consequently, $S$ can be written as
$S=\{x_{n_i,m_i}\ast y_{n_i,m_i}:i\in\omega\}$ for some number
sequences $(m_i)$ and $(n_i)$ with $n_{i+1}>n_i$ for all $i$. It
follows that the sequence $(x_{n_i,m_i})_{i\in\omega}$ converges
to $x_0$ and consequently, the sequence
$T=\{y_{n_i,m_i}\}_{i\in\omega}$ converges to $x_0^{-1}\ast a$.
Since $T\cap \{y_{n_i,m}\}_{m\in{\mathbb N}}\ne\emptyset$ for
every $i$, we conclude that $G$ is an $\alpha_7$-space.
\end{proof}

Lemma~\ref{l4} allows us to prove the following unexpected

\begin{lem}\label{l5} A non-metrizable sequential topological group $G$ with countable
$\mathrm{cs}$-character has a countable $\mathrm{cs}$-network at
the unit, consisting of closed countably compact subsets of $G$.
\end{lem}

\begin{proof} Given a non-metrizable sequential group $G$ with
countable $\mathrm{cs}$-character we can apply
Lemmas~\ref{l2}--\ref{l4} to conclude that $G$ contains no closed
copy of the space $\mathbb L$. Fix a countable $\mathrm{cs}
$-network $\mathcal{N} $ at $ e $, closed under finite
intersections and consisting of closed subspaces of $ G $. We
claim that the collection $\mathcal C\subset\mathcal N$ of all
countably compact subsets $N\in\mathcal N$ forms a
$\mathrm{cs}$-network at $e$ in $G$.

To show this, fix a neighborhood $U\subset G$ of $e$ and a
sequence $(x_n)\subset G$ convergent to $e$. We must find a
countably compact set $M\in\mathcal N$ with $M\subset U$,
containing almost all points $x_n$. Let
$\mathcal{A}=\{A_k:k\in\omega\}$ be the collection of all elements
$N\subset U$ of $\mathcal N$ containing almost all points $x_n$.
Now it suffices to find a number $ n\in\omega $ such that the
intersection $M=\bigcap_{k\le n}A_k$ is countably compact. Suppose
to the contrary, that for every $ n\in\omega$ the set
$\bigcap_{k\le n}A_k$ is not countably compact. Then there exists
a countable closed discrete subspace $ K_{0}\subset A_{0}$ with
$K_0\not\ni e$. Fix a neighborhood $ W_{0}$ of $e$ with $
W_{0}\cap K_{0}=\emptyset $. Since $\mathcal N$ is a
$\mathrm{cs}$-network at $e$, there exists $ k_{1}\in\omega $ such
that $ A_{k_{1}}\subset W_{0}$.

It follows from our hypothesis that there is a countable closed
discrete subspace $K_{1}\subset\bigcap_{k\le k_1}A_k$ with $K_1\ni
e$. Proceeding in this fashion we construct by induction an
increasing number sequence $ (k_{n})_{n\in\omega}\subset\omega $,
a sequence $(K_{n})_{n\in\omega}$ of countable closed discrete
subspaces of $G$, and a sequence $ (W_{n})_{n\in\omega}$ of open
neighborhoods of $ e $ such that $K_{n}\subset\bigcap_{k\le
k_n}A_{k}$, $W_{n}\cap K_{n}=\emptyset$, and $A_{k_{n+1}}\subset
W_{n}$ for all $ n\in\omega $.

It follows from the above construction that $
\{e\}\cup\bigcup_{n\in\omega}K_n$ is a closed copy of the space $
\mathbb{L} $ which is impossible.
\end{proof}

\section*{Proofs of Main Results}

\subsection*{Proof of Proposition~\ref{prop1}} The first three items
can be easily derived from the corresponding definitions. To prove
the fourth item observe that for any $\mathrm{cs}^*$-network
$\mathcal N$ at a point $x$ of a topological space $X$, the family
$\mathcal N'=\{\cup\mathcal F:\mathcal F\subset\mathcal N\}$ is an
$\mathrm{sb}$-network at $x$.

The proof of fifth item is more tricky. Fix any
$\mathrm{cs}^*$-network $\mathcal N$ at a point $x\in X$ with
$|\mathcal N|\le\mathrm{cs}^*_\chi(X)$. Let
$\lambda=\mathrm{cof}(|\mathcal N|)$ be the cofinality of the
cardinal $|\mathcal N|$ and write $\mathcal
N=\bigcup_{\alpha<\lambda}\mathcal N_\alpha$ where $\mathcal
N_\alpha\subset\mathcal N_\beta$ and $|\mathcal
N_\alpha|<|\mathcal N|$ for any ordinals $\alpha\le\beta<\lambda$.
Consider the family $\mathcal M=\{\cup\mathcal C:\mathcal
C\in[\mathcal N_\alpha]^{\le\omega},\; \alpha<\lambda\}$ and
observe that $|\mathcal M|\le\lambda\cdot
\sup\{|[\kappa]^{\le\omega}|:\kappa<|\mathcal N|\}$ where
$[\kappa]^{\le\omega}=\{A\subset\kappa:|A|\le\aleph_0\}$. It rests
to verify that $\mathcal M$ is a $\mathrm{cs}$-network at $x$.

Fix a neighborhood $U\subset X$ of $x$ and a sequence $S\subset X$
convergent to $x$. For every $\alpha<\lambda$ choose a countable
subset $\mathcal C_\alpha\subset\mathcal N_\alpha$ such that
$\cup\mathcal C_\alpha\subset U$ and $S\cap(\cup\mathcal
C_\alpha)=S\cap(\cup\{N\in\mathcal N_\alpha: N\subset U\})$. It
follows that $\cup\mathcal C_\alpha\in\mathcal M$. Let
$S_\alpha=S\cap(\cup\mathcal C_\alpha)$ and observe that
$S_\alpha\subset S_\beta$ for $\alpha\le\beta<\lambda$. To finish
the proof it suffices to show that $S\setminus S_\alpha$ is finite
for some $\alpha<\lambda$. Then the element $\cup\mathcal
C_\alpha\subset U$ of $\mathcal M$ will contain almost all members
of the sequence $S$.

Separately, we shall consider the cases of countable and
uncountable $\lambda$. If $\lambda$ is uncountable, then it has
uncountable cofinality and consequently, the transfinite sequence
$(S_\alpha)_{\alpha<\lambda}$ eventually stabilizes, i.e., there
is an ordinal $\alpha<\lambda$ such that $S_\beta=S_\alpha$ for
all $\beta\ge\alpha$. We claim that the set $S\setminus S_\alpha$
is finite. Otherwise, $S\setminus S_\alpha$ would be a sequence
convergent to $x$ and there would exist an element $N\in\mathcal
N$ with $N\subset U$ and infinite intersection $N\cap (S\setminus
S_\alpha)$. Find now an ordinal $\beta\ge\alpha$ with
$N\in\mathcal N_\beta$ and observe that $S\cap N\subset
S_\beta=S_\alpha$ which contradicts to the choice of $N$.

If $\lambda$ is countable and $S\setminus S_\alpha$ is infinite
for any $\alpha<\lambda$, then we can find an infinite
pseudo-intersection $T\subset S$ of the decreasing sequence
$\{S\setminus S_\alpha\}_{\alpha<\lambda}$. Note that $T\cap
S_\alpha$ is finite for every $\alpha<\lambda$.
 Since sequence $T$  converges to $x$, there is an element
 $N\in\mathcal N$ such that $N\subset U$ and $N\cap T$ is infinite.
Find $\alpha<\lambda$ with $N\in\mathcal N_\alpha$ and observe
that $N\cap S\subset S_\alpha$. Then $N\cap T\subset N\cap T\cap
S_\alpha\subset T\cap S_\alpha$ is finite, which contradicts to
the choice of $N$.

\subsection*{Proof of Proposition~\ref{cs=cs*}} Let $X$ be a
topological space and fix a point $x\in X$.

(1) Suppose that $\mathrm{cs}^*_\chi(X)<\mathfrak p$ and fix a
$\mathrm{cs}^*$-network $\mathcal N$ at the point $x$ such that
$|\mathcal N|<\mathfrak p$. Without loss of generality, we can
assume that the family $\mathcal N$ is closed under finite unions.
We claim that $\mathcal N$ is a $\mathrm{cs}$-network at $x$.
Assuming the converse we would find a neighborhood $U\subset X$ of
$x$  and a sequence $S\subset X$ convergent to $x$ such that
$S\setminus N$ is infinite for any element $N\in\mathcal N$ with
$N\subset U$. Since $\mathcal N$ is closed under finite unions,
the family $\mathcal F=\{S\setminus N:N\in\mathcal N,\; N\subset
U\}$ is closed under finite intersections. Since $|\mathcal
F|\le|\mathcal N|<\mathfrak p$, the family $\mathcal F$ has an
infinite pseudo-intersection $T\subset S$. Consequently, $T\cap N$
is finite for any $N\in\mathcal N$ with $N\subset U$. But this
contradicts to the facts that $T$ converges to $x$ and $\mathcal
N$ is a $\mathrm{cs}^*$-network at $x$.

The items (2) and (3) follow from Propositions~\ref{prop1}(5) and
\ref{cs=cs*}(1). The item (4) follows from (1,2) and the
inequality $\chi(X)\le\mathfrak c$ holding for any countable
topological space $X$.

Finally, to derive (5) from (3) use the well-known fact that under
GCH, $\lambda^{\aleph_0}\le\kappa$ for any infinite cardinals
$\lambda<\kappa$, see \cite[9.3.8]{HJ}.

\subsection*{Proof of Theorem~\ref{main}} Suppose that $G$ is a
non-metrizable sequential group with countable
$\mathrm{cs}^*$-character. By Proposition~\ref{cs=cs*}(1),
$\mathrm{cs}_\chi(G)=\mathrm{cs}^*_\chi(G)\le\aleph_0$.

First we show that each countably compact subspace $ K $ of $ G $
is first-countable. The space $K$, being countably compact in the
sequential space $G$, is sequentially compact and so are the sets
$K^{-1}K$ and $(K^{-1}K)^{-1}(K^{-1}K)$ in $G$. The sequential
compactness of $K^{-1}K$ implies that it is an $\alpha_7$-space.
Since
$\mathrm{cs}_\chi((K^{-1}K)^{-1}(K^{-1}K))\le\mathrm{cs}_\chi(G)\le\aleph_0$
we may apply Lemmas~\ref{l3} and \ref{l2} to conclude that the
space $K^{-1}K$ has countable $\mathrm{sb}$-character and $K$ has
countable character.

Next, we show that $ G $ contains an open $\mathcal{MK}_\omega
$-subgroup. By Lemma \ref{l5}, $ G $ has a countable $\mathrm{cs}
$-network $\mathcal{K}$ consisting of countably compact subsets.
Since the group product of two countably compact subspaces in $ G
$ is countably compact, we may assume that $\mathcal{K}$ is closed
under finite group products in $G$. We can also assume that
$\mathcal{K} $ is closed under the inversion, i.e. $ K^{-1}\in
\mathcal{K} $ for any $ K\in\mathcal{K} $. Then $H=\cup K$ is a
subgroup of $G$. It follows that this subgroup is a sequential
barrier at each of its points, and thus is open-and-closed in $G$.
We claim that the topology on $ H $ is inductive with respect to
the cover $\mathcal{K} $. Indeed, consider some $ U\subset H $
such that $ U\cap K $ is open in $ K $ for every $ K\in\mathcal{K}
$. Assuming that $U$ is not open in $H$ and using the
sequentiality of $H$, we would find a point $ x\in U $ and a
sequence $ (x_n)_{n\in\omega}\subset H\setminus U $ convergent to
$ x $. It follows that there are elements $ K_1 ,K_2
\in\mathcal{K} $ such that $ x\in K_1 $ and $K_2$ contains almost
all members of the sequence $(x^{-1}x_n)$. Then the product $
K=K_1 K_2 $ contains almost all $x_n$ and the set $ U\cap K $,
being an open neighborhood of $x$ in $ K $, contains almost all
members of the sequence $(x_n)$, which is a contradiction.

As it was proved before each $ K\in\mathcal{K} $ is
first-countable, and consequently $ H $ has countable
pseudocharacter, being the countable union of first countable
subspaces. Then $H$ admits a continuous metric. Since any
continuous metric on a countably compact space generates its
original topology, every $ K\in\mathcal{K} $ is a metrizable
compactum, and consequently $H$ is an
$\mathcal{MK}_\omega$-subgroup of $ G $.

Since $H$ is an open subgroup of  $G$, $G$ is homeomorphic to
$H\times D$ for some discrete space $D$.

\subsection*{Proof of Theorem~\ref{cardinal}} Suppose $G$ is a
non-metrizable sequential topological group with countable
$\mathrm{cs}^*$-character. By Theorem~\ref{main}, $G$ contains an
open $\mathcal{MK}_\omega$-subgroup $H$ and is homeomorphic to the
product $H\times D$ for some discrete space $D$ . This implies
that $G$ has point-countable $k$-network. By a result of Shibakov
\cite{Shi}, each sequential topological group with point-countable
$k$-network and sequential order $<\omega_1$ is metrizable.
Consequently, $\mathrm{so}(G)=\omega_1$. It is clear that
$\psi(G)=\psi(H)\le \aleph_0$, $\chi(G)=\chi(H)$,
$\mathrm{sb}_\chi(G)=\mathrm{sb}_\chi(H)$ and
${ib}(G)=c(G)=d(G)=l(G)=e(G)={nw}(G)={knw}(G)=|D|\cdot\aleph_0$.

To finish the proof it rests to show that
$\mathrm{sb}_\chi(H)=\chi(H)=\mathfrak d$. It follows from
Lemmas~\ref{l2} and \ref{l3} that the group $H$, being
non-metrizable, is not $\alpha_7$ and thus contains a copy of the
sequential fan $S_\omega$. Then $\mathfrak
d=\chi(S_\omega)=\mathrm{sb}_\chi(S_\omega)\le\mathrm{sb}_\chi(H)\le\chi(H)$.
 To prove that
$\chi(H)\le\mathfrak d$ we shall apply a result of K.~Sakai
\cite{Sa} asserting that the space $\mathbb{R}^\infty\times Q$
contains a closed topological copy of each
$\mathcal{MK}_\omega$-space and the well-known equality
$\chi(\mathbb{R}^\infty\times Q)=\chi(\mathbb{R}^\infty)=\mathfrak
d$ (following from the fact that $\mathbb{R}^\infty$ carries the
box-product topology, see \cite[Ch.II, Ex.12]{Sch}).

\subsection*{Proof of Theorem~\ref{dowker}}
First we describe two general constructions producing
topologically homogeneous sequential spaces. For a locally compact
space $Z$ let $\alpha Z=Z\cup\{\infty\}$ be the one-point
extension of $Z$ endowed with the topology whose neighborhood base
at $\infty$ consists of the sets $\alpha Z\setminus K$ where $K$
is a compact subset of $Z$. Thus for a non-compact locally
compacts space $Z$ the space $\alpha Z$ is noting else but the
one-point compactification of $Z$. Denote by
$2^\omega=\{0,1\}^\omega$ the Cantor cube.

Consider the subsets
\begin{align*}\Xi(Z)=&\{(c,(z_i)_{i\in\omega})\in 2^\omega\times
(\alpha Z)^\omega: z_i=\infty\mbox{ for all but finitely many
indices $i\}$ and}\\ \Theta(Z)=&\{(c,(z_i)_{i\in\omega})\in
2^\omega\times (\alpha Z)^\omega: \exists n\in\omega\mbox{ such
that $z_i\ne\infty$ if and only if $i<n$}\}. \end{align*} Observe
that $\Theta(Z)\subset\Xi(Z)$.

Endow the set $\Xi(Z)$ (resp. $\Theta(Z)$) with the strongest
topology generating the Tychonov product topology on each compact
subset from the family $\mathcal{K}_\Xi$ (resp.
$\mathcal{K}_\Theta$), where
\smallskip

$\mathcal{K}_\Xi=\{2^\omega\times \prod_{i\in\omega}C_i: C_i\mbox{
are compact subsets of $\alpha Z$ and almost all
$C_i=\{\infty\}\}$;}$
\smallskip

$\mathcal{K}_\Theta=\{2^\omega\times \prod_{i\in\omega}C_i:\;
\exists i_0\in\omega$ such that $C_{i_0}=\alpha Z$,
$C_i=\{\infty\}$ for all $i>i_0$ and\newline \vskip-13pt

 {}\hskip120pt $C_i$ is a compact subsets of $Z$ for every
$i<i_0\}$.
\smallskip

\begin{lem}\label{homogeneous} Suppose $Z$ is a zero-dimensional
locally metrizable locally compact space. Then
\begin{enumerate}
\item the spaces $\Xi(Z)$ and $\Theta(Z)$ are topologically homogeneous;
\item $\Xi(Z)$ is a regular zero-dimensional $k_\omega$-space
while $\Theta(Z)$ is a totally disconnected $k$-space;
\item if $Z$ is Lindel\"of, then $\Xi(Z)$ and $\Theta(Z)$ are zero-dimensional
$\mathcal{MK}_\omega$-spaces with\newline
$\chi(\Xi(Z))=\chi(\Theta(Z))\le\mathfrak d$;
\item $\Xi(Z)$ and $\Theta(Z)$ contain copies of the space
$\alpha Z$ while $\Theta(Z)$ contains a closed copy of $Z$;
\item $\mathrm{cs}^*_\chi(\Xi(Z))=\mathrm{cs}^*_\chi(\Theta(Z))=\mathrm{cs}^*_\chi(\alpha Z)$,
$\mathrm{cs}_\chi(\Xi(Z))=\mathrm{cs}_\chi(\Theta(Z))=\mathrm{cs}_\chi(\alpha
Z)$,\newline $\mathrm{sb}_\chi(\Theta(Z))=\mathrm{sb}_\chi(\alpha
Z)$, and $\psi(\Xi(Z))=\psi(\Theta(Z))=\psi(\alpha Z)$;
\item the spaces $\Xi(Z)$ and $\Theta(Z)$ are sequential if and only if
$\alpha Z$ is sequential;
\item if $Z$ is not countably compact, then
$\Xi(Z)$ contains a closed copies of $S_2$ and $S_\omega$ and
$\Theta(Z)$ contains a closed copy of $S_2$.
\end{enumerate}
\end{lem}

\begin{proof}
(1) First we show that the space $\Xi(Z)$ is topologically
homogeneous.

Given two points $(c,(z_i)_{i\in\omega}),
(c',(z'_i)_{i\in\omega})$ of $\Xi(Z)$  we have to find a
homeomorphism $h$ of $\Xi(Z)$ with
$h(c,(z_i)_{i\in\omega})=(c',(z'_i)_{i\in\omega})$. Since the
Cantor cube $2^\omega$ is topologically homogeneous, we can assume
that $c\ne c'$. Fix any disjoint closed-and-open neighborhoods
$U,U'$ of the points $c,c'$ in $2^\omega$, respectively.

Consider the finite sets $I=\{i\in\omega: z_i\ne\infty\}$ and
$I'=\{i\in\omega: z_i'\ne\infty\}$. Using the zero-dimensionality
and the local metrizability of $Z$, for each $i\in I$ (resp. $i\in
I'$) fix an open compact metrizable neighborhood $U_i$ (resp.
$U_i'$) of the point $z_i$ (resp. $z'_i$) in $Z$. By the classical
Brouwer Theorem \cite[7.4]{Ke}, the products $U\times \prod_{i\in
I}U_i$ and $U'\times \prod_{i\in I'}U'_i$, being zero-dimensional
compact metrizable spaces without isolated points, are
homeomorphic to the Cantor cube $2^\omega$. Now the topological
homogeneity of the Cantor cube implies the existence of a
homeomorpism $f:U\times \prod_{i\in I}U_i\to U'\times \prod_{i\in
I'}U'_i$ such that $f(c,(z_i)_{i\in I})=(c',(z'_i)_{i\in I'})$.
Let

$W=\{(x,(x_i)_{i\in\omega})\in \Xi(Z): x\in U,\; x_i\in U_i$ for
all $i\in I\}$ and

$W'=\{(x',(x'_i)_{i\in\omega})\in\Xi(Z): x'\in U',\; x'_i\in U'_i$
for all $i\in I'\}$.

\noindent It follows that $W,W'$ are disjoint open-and-closed
subsets of $\Xi(Z)$. Let $\chi:\omega\setminus I'\to \omega
\setminus I$ be a unique monotone bijection.

Now consider the homeomorphism $\tilde f:W\to W'$ assigning to a
sequence $(x,(x_i)_{i\in\omega})\in W$ the sequence
$(x',(x'_i)_{i\in\omega})\in W'$ where $(x',(x'_i)_{i\in
I'})=f(x,(x_i)_{i\in I})$ and $x'_i=x_{\chi(i)}$ for $i\notin I'$.
Finally, define a homeomorphism $h$ of $\Xi(Z)$ letting $$ h(x)=
  \begin{cases}
    x & \text{if $x\notin W\cup W'$;} \\
    \tilde f(x) & \text{if $x\in W$;}\\
    \tilde f^{-1}(x) & \text{if $x\in W'$}
  \end{cases}
  $$
and observe that
$h(c,(z_i)_{i\in\omega})=(c',(z'_i)_{i\in\omega})$ which proves
the topological homogeneity of the space $\Xi(Z)$.

Replacing $\Xi(Z)$ by $\Theta(Z)$ in the above proof, we shall get
a proof of the topological homogeneity of $\Theta(Z)$.

The items (2--4) follow easily from the definitions of the spaces
$\Xi(Z)$ and $\Theta(Z)$, the zero-dimensionality of $\alpha Z$,
and known properties of $k_\omega$-spaces, see \cite{FS} (to find
a closed copy of $Z$ in $\Theta(Z)$ consider the closed embedding
$e:Z\to\Theta(Z)$, $e:z\mapsto(z,z_0,z,\infty,\infty,\dots)$,
where $z_0$ is any fixed point of $Z$).

To prove (5) apply Proposition~\ref{prop3}(6,8,9,10). (To
calculate the $\mathrm{cs}^*$-, $\mathrm{cs}$-, and
$\mathrm{sb}$-characters of $\Theta(Z)$, observe that almost all
members of any sequence $(a_n)\subset \Theta(Z)$ convergent to a
point $a=(c,(z_i))\in\Theta(Z)$ lie in the compactum
$2^\omega\times \prod_{i\in\omega}C_i$, where $C_i$ is a clopen
neighborhood of $z_i$ if $z_i\ne\infty$, $C_i=\alpha Z$ if
$i=\min\{j\in\omega: z_j=\infty\}$ and $C_i=\{\infty\}$ otherwise.
By Proposition~\ref{prop3}(6), the $\mathrm{cs}^*$-,
$\mathrm{cs}$-, and $\mathrm{sb}$-characters of this compactum are
equal to the corresponding characters of $\alpha Z$.)

(6) Since the spaces $\Xi(Z)$ and $\Theta(Z)$ contain a copy of
$\alpha Z$, the sequentiality of $\Xi(Z)$ or $\Theta(Z)$ implies
the sequentiality of $\alpha Z$. Now suppose conversely that the
space $\alpha Z$ is sequential. Then each compactum
$K\in\mathcal{K}_\Xi\cup\mathcal{K}_\Theta$ is sequential since a
finite product of sequential compacta is sequential, see
\cite[3.10.I(b)]{En1}. Now  the spaces $\Xi(Z)$ and $\Theta(Z)$
are sequential because they carry the inductive topologies with
respect to the covers $\mathcal K_\Xi$, $\mathcal K_\Theta$ by
sequential compacta.

(7) If $Z$ is not countably compact, then it contains a countable
closed discrete subspace $S\subset Z$ which can be thought as a
sequence convergent to $\infty$ in $\alpha Z$. It is easy to see
that $\Xi(S)$ (resp. $\Theta(S)$) is a  closed subset of $\Xi(Z)$
(resp. $\Theta(Z)$). Now it is quite easy to find closed copies of
$S_2$ and $S_\omega$ in $\Xi(S)$ and a closed copy of $S_2$ in
$\Theta(S)$.
\end{proof}

With Lemma~\ref{homogeneous} at our disposal,  we are able to
finish the proof of Theorem~\ref{dowker}. To construct the
examples satisfying the conditions of Theorem~\ref{dowker}(2,3),
assume $\mathfrak b=\mathfrak c$ and use
Proposition~\ref{yakovlev} to find
 a locally
compact locally countable space $Z$ whose one-point
compactification $\alpha Z$ is  sequential and satisfies
$\aleph_0=\mathrm{sb}_\chi(\alpha Z)<\psi(\alpha Z)=\mathfrak c$.
Applying Lemma~\ref{homogeneous} to this space $Z$, we conclude
that the topologically homogeneous $k$-spaces $X_2=\Xi(Z)$ and
$X_3=\Theta(Z)$ give us required examples.

The example of a countable topologically homogeneous
$k_\omega$-space $X_1$ with $\mathrm{sb}_\chi(X_1)<\chi(X_1)$ can
be constructed by analogy with the space $\Theta(\mathbb{N})$
(with that difference that there is no necessity to involve the
Cantor cube) and is known in topology as the Ankhangelski-Franklin
space, see \cite{AF}. We briefly remind its construction. Let
$S_0=\{0,\frac1n:n\in\mathbb{N}\}$ be a convergent sequence and
consider the countable space $X_1=\{(x_i)_{i\in\omega}\in
S_0^\omega: \exists n\in\omega\mbox{ such that $x_i\ne 0$ iff
$i<n$}\}$ endowed with the strongest topology inducing the product
topology on each compactum $\prod_{i\in\omega}C_i$ for which there
is $n\in\omega$ such that $C_n=S_0$, $C_i=\{0\}$ if $i>n$, and
$C_i=\{x_i\}$ for some $x_i\in S_0\setminus \{0\}$ if $i<n$. By
analogy with the proof of Lemma~\ref{homogeneous} it can be shown
that $X_1$ is a topologically homogeneous $k_\omega$-space with
$\aleph_0=\mathrm{sb}_\chi(X_1)<\chi(X_1)=\mathfrak d$ and
$\mathrm{so}(X_1)=\omega$.

\subsection*{Proof of Proposition~\ref{sb}} The equivalences
$(1)\Leftrightarrow(2)\Leftrightarrow(3)$ were proved by Lin
\cite[3.13]{Lin} in terms of (universally) $csf$-countable spaces.
To prove the other equivalences apply

\begin{lem}\label{alpha4} A Hausdorff topological space $X$ is an
$\alpha_4$-space provided one of the following conditions is
satisfied:
\begin{enumerate}
\item $X$ is a Fr\'echet-Urysohn $\alpha_7$-space;
\item $X$ is a Fr\'echet-Urysohn countably compact space;
\item $\mathrm{sb}_\chi(X)<\mathfrak p$;
\item $\mathrm{sb}_\chi(X)<\mathfrak d$, each point of $X$ is regular
$G_\delta$, and $X$ is $c$-sequential.
\end{enumerate}
\end{lem}

\begin{proof} Fix any point
$x\in X$ and a countable family $\{S_n\}_{n\in\omega}$ of
sequences convergent to $x$ in $X$. We have to find a sequence
$S\subset X\setminus\{x\}$ convergent to $x$ and meeting
infinitely many sequences $S_n$. Using the countability of the set
$\bigcup_{n\in\omega}S_n$ find a decreasing sequence
$(U_n)_{n\in\omega}$ of closed neighborhoods of $x$ in $X$ such
that
$\big(\bigcap_{n\in\omega}U_n\big)\cap\big(\bigcup_{n\in\omega}
S_n)=\{x\}$. Replacing each sequence $S_n$ by its subsequence
$S_n\cap U_n$, if necessary, we can assume that $S_n\subset U_n$.

(1) Assume that $X$ is a Fr\'echet-Urysohn $\alpha_7$-space. Let
$A=\{a\in X: a$ is the limit of a convergent sequence $S\subset X$
meeting infinitely many sequences $S_n\}$. It follows from our
assumption on $(S_n)$ and $(U_n)$ that
$A\subset\bigcap_{n\in\omega}U_n$.

It suffices to consider the non-trivial case when $x\notin A$. In
this case $x$ is a cluster point of $A$ (otherwise $X$ would be
not $\alpha_7$). Since $X$ is Fr\'echet-Urysohn, there is a
sequence $(a_n)\subset A$ convergent to $x$. By the definition of
$A$, for every $n\in\omega$ there is a sequence $T_n\subset X$
convergent to $a$ and meeting infinitely many sequences $S_n$.
Without loss of generality, we can assume that
$T_n\subset\bigcup_{i>n}S_i$ (because $a\in A\setminus\{x\}$ and
thus $a\notin \bigcup_{n\in\omega}S_n$). It is easy to see that
$x$ is a cluster point of the set $\bigcup_{n\in\omega}T_n$. Since
$X$ is Fr\'echet-Urysohn, there is a sequence
$T\subset\bigcup_{n\in\omega}T_n$ convergent to $x$.

Now it rests to show that the set $T$ meets infinitely many
sequences $S_n$. Assuming the converse we would find $n\in\omega$
such that $T\subset\bigcup_{i\le n}S_n$. Then
$T\subset\bigcup_{i\le n}T_n$ which is not possible since
$\bigcup_{i\le n}T_i$ is a compact set failing to contain the
point $x$.

(2) If $X$ is Fr\'echet-Urysohn and countably compact, then it is
sequentially compact and hence $\alpha_7$, which allows us to
apply the previous item.

(3) Assume that $\mathrm{sb}_\chi(X)<\mathfrak p$ and let
$\mathcal N$ be a $\mathrm{sb}$-network at $x$ of size $|\mathcal
N|<\mathfrak p$. Without loss of generality, we can assume that
the family $\mathcal{N}$ is closed under finite intersections. Let
$S=\bigcup_{n\in\omega} S_n$ and $F_{N,n}=N\cap (\bigcup_{i\ge
n}S_i)$ for $N\in\mathcal{N}$ and $n\in\omega$. It is easy to see
that the family $\mathcal F=\{F_{N,n}:N\in\mathcal{N},\;
n\in\omega\}$ consists of infinite subsets of $S$, has size
$|\mathcal F|<\mathfrak p$, and is closed under finite
intersection. Now the definition of the small cardinal $\mathfrak
p$ implies that this family $\mathcal{F}$ has an infinite
pseudo-intersection $T\subset S$. Then $T$ is a sequence
convergent to $x$ and intersecting infinitely many sequences
$S_n$. This shows that $X$ is an $\alpha_4$-space.

(4) Assume that the space $X$ is $c$-sequential, each point of $X$
is regular $G_\delta$, and $\mathrm{sb}_\chi(X)<\mathfrak d$. In
this case we can choose the sequence $(U_n)$ to satisfy
$\bigcap_{n\in\omega}U_n=\{x\}$. Fix an $\mathrm{sb}$-network
$\mathcal N$ at $x$ with $|\mathcal N|<\mathfrak d$. For every
$n\in\omega$ write $S_n=\{x_{n,i}:i\in\mathbb{N}\}$. For each
sequential barrier $N\in\mathcal N$ find a function
$f_N:\omega\to\mathbb{N}$ such that $x_{n,i}\in N$ for every
$n\in\omega$ and $i\ge f_N(n)$. The family of functions
$\{f_N:N\in\mathcal N\}$ has size $<\mathfrak d$ and hence is not
cofinal in $\mathbb{N}^\omega$. Consequently, there is a function
$f:\omega\to\mathbb{N}$ such that $f\not\le f_N$ for each
$N\in\mathcal{N}$. Now consider the sequence
$S=\{x_{n,f(n)}:n\in\omega\}$. We claim that $x$ is a cluster
point of $S$. Indeed, given any neighborhood $U$ of $x$, find a
sequential barrier $N\in\mathcal{N}$ with $N\subset U$. Since
$f\not\le f_N$, there is $n\in\omega$ with $f(n)>f_N(n)$. It
follows from the choice of the function $f_N$ that $x_{n,f(n)}\in
N\subset U$.

Since $S\setminus U_n$ is finite for every $n$,
$\{x\}=\bigcap_{n\in\omega}U_n$ is a unique cluster point of $S$
and thus $\{x\}\cup S$ is a closed subset of $X$. Now the
$c$-sequentiality of $X$ implies the existence of a sequence
$T\subset S$ convergent to $x$. Since $T$ meets infinitely many
sequences $S_n$, the space $X$ is $\alpha_4$.
\end{proof}

\subsection*{Proof of Proposition~\ref{chi}} Suppose a space $X$
has countable $\mathrm{cs}^*$-character. The implications
$(1)\Rightarrow(2,3,4,5)$ are trivial. The equivalence
$(1)\Leftrightarrow(2)$ follows from Proposition~\ref{prop1}(2).
To show that $(3)\Rightarrow(2)$, apply Lemma~\ref{alpha4} and
Proposition~\ref{sb}($3\Rightarrow1$).

To prove that $(4)\Rightarrow(2)$ it suffices to apply
Proposition~\ref{sb}($4\Rightarrow1$) and observe that $X$ is
Fr\'echet-Urysohn provided $\chi(X)<\mathfrak p$ and $X$ has
countable tightness. This can be seen as follows.

 Given a subset $A\subset X$ and a
point $a\in\bar A$ from its closure, use the countable tightness
of $X$ to find a countable subset $N\subset A$ with $a\in\bar N$.
Fix any neighborhood base $\mathcal B$ at $x$ of size $|\mathcal
B|<\mathfrak p$. We can assume that $\mathcal B$ is closed under
finite intersections. By the definition of the small cardinal
$\mathfrak p$, the family $\{B\cap N:B\in\mathcal B\}$ has
infinite pseudo-intersection $S\subset N$. It is clear that
$S\subset A$ is a sequence convergent to $x$, which proves that
$X$ is Fr\'echet-Urysohn.

$(5)\Rightarrow (2)$. Assume that $X$ is a sequential space
containing no closed copies of $S_\omega$ and $S_2$ and such that
each point of $X$ is regular $G_\delta$. Since $X$ is sequential
and contains no closed copy of $S_2$, we may apply Lemma 2.5
\cite{Lin} to conclude that $X$ is Fr\'echet-Urysohn. Next,
Theorem 3.6 of \cite{Lin} implies that $X$ is an $\alpha_4$-space.
Finally apply Proposition~\ref{sb} to conclude that $X$ has
countable $\mathrm{sb}$-character and, being Fr\'echet-Urysohn, is
first countable.

The final implication $(6)\Rightarrow(2)$ follows from
$(5)\Rightarrow(2)$ and the well-known equality
$\chi(S_\omega)=\chi(S_2)=\mathfrak d$.

\subsection*{Proof of Proposition~\ref{dyadic}} \par

The first item of this proposition follows from
Proposition~\ref{chi}($3\Rightarrow1$) and the observation that
each Fr\'echet-Urysohn countable compact space, being sequentially
compact, is $\alpha_7$.

Now suppose that $X$ is a dyadic compact with
$\mathrm{cs}^*_\chi(X)\le\aleph_0$. If $X$ is not metrizable, then
it contains a copy of the one-point compactification $\alpha D$ of
an uncountable discrete space $D$, see \cite[3.12.12(i)]{En1}.
Then $\mathrm{cs}^*_\chi(\alpha
D)\le\mathrm{cs}^*_\chi(X)\le\aleph_0$ and by the previous item,
the space $\alpha D$, being Fr\'echet-Urysohn and compact, is
first-countable, which is a contradiction.

\subsection*{Proof of Proposition~\ref{alphaD}}
Let $D$ be a discrete space.

(1) Let $\kappa=\mathrm{cs}^*_\chi(\alpha D)$ and $\lambda_1$
($\lambda_2$) is the smallest weight of a (regular
zero-dimensional) space $X$ of size $|X|=|D|$, containing no
non-trivial convergent sequence. To prove the first item of
proposition~\ref{alphaD} it suffices to verify that
$\lambda_2\le\kappa\le\lambda_1$. To show that
$\lambda_2\le\kappa$, fix any $\mathrm{cs}^*$-network $\mathcal N$
at the unique non-isolated point $\infty$ of $\alpha D$ of size
$|\mathcal N|\le\kappa$. The algebra $\mathcal A$ of subsets of
$D$ generated by the family $\{D\setminus N:N\in\mathcal N\}$ is a
base of some zero-dimensional topology $\tau$ on $D$ with
$w(D,\tau)\le\kappa$. We claim that the space $D$ endowed with
this topology contains no infinite convergent sequences. To get a
contradiction, suppose that $S\subset D$ is an infinite sequence
convergent to a point $a\in D\setminus S$. Then $S$ converges to
$\infty$ in $\alpha D$ and hence, there is an element
$N\in\mathcal N$ such that $N\subset \alpha D\setminus\{a\}$ and
$N\cap S$ is infinite. Consequently, $U=D\setminus N$ is a
neighborhood of $a$ in the topology $\tau$ such that $S\setminus
U$ is infinite which contradicts to the fact that $S$ converges to
$a$. Now consider the equivalence relation $\sim$ on $D$: $x\sim
y$ provided for every $U\in\tau$\; $(x\in U)\Leftrightarrow(y\in
U)$. Since the space $(D,\tau)$ has no infinite convergent
sequences, each equivalence class $[x]_\sim\subset D$ is finite
(because it carries the anti-discrete topology). Consequently, we
can find a subset $X\subset D$ of size $|X|=|D|$ such that
$x\not\sim y$ for any distinct points $x,y\in X$. Clearly that
$\tau$ induces a zero-dimensional topology on $X$. It rests to
verify that this topology is $T_1$. Given any two distinct point
$x,y\in X$ use $x\not\sim y$ to find an open set $U\in\mathcal A$
such that either $x\in U$ and $y\notin U$ or $x\notin U$ and $y\in
U$. Since $D\setminus U\in\mathcal A$, in both cases we find an
open set $W\in\mathcal A$ such that $x\in W$ but $y\notin W$. It
follows that $X$ is a $T_1$-space containing no non-trivial
convergent sequence and thus $\lambda_2\le w(X)\le|\mathcal
A|\le\mathcal|N|\le\kappa$.

To show that $\kappa\le\lambda_1$, fix any topology $\tau$ on $D$
such that $w(D,\tau)\le\lambda_1$ and the space $(D,\tau)$
contains no non-trivial convergent sequences. Let $\mathcal B$ be
a base of the topology $\tau$ with $|\mathcal B|\le\lambda_1$,
closed under finite unions. We claim that the collection $\mathcal
N=\{\alpha D\setminus B:B\in\mathcal B\}$ is a
$\mathrm{cs}^*$-network for $\alpha D$ at $\infty$. Fix any
neighborhood $U\subset\alpha D$ of $\infty$ and any sequence
$S\subset D$ convergent to $\infty$. Write
$\{x_1,\dots,x_n\}=\alpha D\setminus U$ and by finite induction,
for every $i\le n$ find a neighborhood $B_i\in\mathcal B$ of $x_i$
such that $S\setminus\bigcup_{j\le i}B_j$ is infinite. Since
$\mathcal B$ is closed under finite unions, the set $N=\alpha
D\setminus(B_1\cup\dots\cup B_n)$ belongs to the family $\mathcal
N$ and has the properties: $N\subset U$ and $N\cap S$ is infinite,
i.e., $\mathcal N$ is a $\mathrm{cs}^*$-network at $\infty$ in
$\alpha D$. Thus $\kappa\le|\mathcal N|\le|\mathcal
B|\le\lambda_1$. This finishes the proof of (1).

An obvious modification of the above argument gives also a proof
of the item (2).

\subsection*{Proof of Proposition~\ref{alphaD1}} Let $D$ be an
uncountable discrete space.

(1) The inequalities
$\aleph_1\cdot\log|D|\le\mathrm{cs}^*_\chi(\alpha
D)\le\mathrm{cs}_\chi(\alpha D)$ follows from
Propositions~\ref{dyadic}(1) and \ref{prop1}(2,4) yielding
$|D|=\chi(\alpha D)=\mathrm{sb}_\chi(\alpha D)\le
2^{\mathrm{cs}_\chi^*(\alpha D)}$. The inequality
$\mathrm{cs}_\chi(\alpha D)\le \mathfrak
c\cdot\mathrm{cof}([\log|D|]^{\le\omega})$ follows from
proposition~\ref{alphaD}(2) and the observation that the product
$\{0,1\}^{\log|D|}$ endowed with the $\aleph_0$-box product
topology has weight $\le\mathfrak
c\cdot\mathrm{cof}([\log|D|]^{\le\omega})$. Under the {\em
$\aleph_0$-box product topology} on $\{0,1\}^\kappa$ we understand
the topology generated by the base consisting of the sets
$\{f\in\{0,1\}^\kappa:f|C=g|C\}$ where $g\in\{0,1\}^\omega$ and
$C$ is a countable subset of $\kappa$.

The item (2) follows from (1) and the equality
$\aleph_1\cdot\log\kappa=2^{\aleph_0}\cdot\min\{\kappa,(\log\kappa)^\omega\}$
holding under GCH for any infinite cardinal $\kappa$, see
\cite[9.3.8]{HJ}

\address{Department of Mathematics, Ivan Franko Lviv National University,
Universytetska 1, Lviv, 79000, Ukraine}
\email{tbanakh@franko.lviv.ua}

\newpage

\begin{thebibliography}{}

\bibitem[Ar$_1$]{Ar} A.V.~Arkhangelski, {\em Maps and spaces},
Uspekhi Mat. Nauk {\bf 21}:4 (1966), 133--184 (in Russian).

\bibitem[Ar$_2$]{Ar2} A.V.~Arkhangelski, {\em The frequency spectrum of
a topological space and the classification of spaces}, Soviet
Math. Dokl. {\bf 13} (1972), 265--268 (in Russian).

\bibitem[Ar$_3$]{Ar4} A.V.~Arkhangelski, {\em Structure and
classification of topological spaces and cardinal invariants},
Russian Math. Surveys {\bf 33} (1978), 33--96 (in Russian).

\bibitem[Ar$_4$]{Ar3} A.V.~Arkhangelski, {\em The frequency spectrum of
a topological space and the product operation}, Trans. Moscow
Math. Soc. {\bf 2} (1981), 163--200.

\bibitem[AF]{AF} A.V.~Arkhangelski, S.P.~Franklin, {\em Ordinal
invariants for topological spaces}, Michigan Math. J. {\bf 15}
(1968), 313--320; 506.

\bibitem[Ba$_1$]{Ba1} T.~Banakh, {\em On topological groups containing a
 Fr\'echet-Urysohn fan}, Matem. Studii {\bf 9}:2 (1998), 149--154.

\bibitem[Ba$_2$]{Ba3}
T.~Banakh, {\em Topological classification of strong duals to
nuclear (LF)-spaces}, Studia Math. {\bf 138}:3 (2000), 201--208.

\bibitem[Ba$_3$]{Ba2} T.~Banakh, {\em Topological classification of zero-dimensional
$\mathcal M_\omega$-groups},  Matem. Studii. {\bf 15}:1 (2001),
109--112.

\bibitem[Co]{Co} W. Comfort, {\em Topological groups}, in:
K.Kunen, J.E.Vaughan (eds.) {\em Handbook of Set-theoretic
Topology},(Elsevier Science Publishers B.V., 1984), 1143--1263.

\bibitem[vD]{vD} E.K.~van Douwen, {\em The integers and Topology}, in:
K.Kunen, J.E.Vaughan (eds.), {\em Handbook of Set-Theoretic
Topology} (North-Holland, Amsterdam, 1984), 111--167.

\bibitem[En$_1$]{En1} R.~Engelking, {\em General Topology}
(Mir, Moskow, 1986), (in Russian).

\bibitem[En$_2$]{En2} R.~Engelking, {\em Theory of Dimensions, Finite and
Infinite} (Heldermann Verlag, 1995).



\bibitem[FST]{FS} S.P.~Franklin, B.V.~Smith Thomas,
{\em A~survey of $k_{\omega}$-spaces}, Topology Proc. {\bf 2}
(1977), 111--124.

\bibitem[HJ]{HJ}
K.~Hrbacek, T.~Jech. {\em Introduction to Set Theory}, (Marcel
Dekker, 1999).

\bibitem[Lin]{Lin} S.~Lin, {\em A note on Arens space and sequential fan},
Topology Appl. {\bf 81} (1997), 185-196.

\bibitem[Ke]{Ke} A.~Kechris, {\em Classical Descriptive Set
Theory} (Springer, 1995).



\bibitem[MS]{MS} D.A.~Martin, R.M.~Solovay. {\em Internal Cohen
Extensions}, Ann. Math. Logic. {\bf 2} (1970), 143--178.

\bibitem[Na]{Nag} J.~Nagata. {\em Generalized Metric Spaces, I}. in:
K.Morita, J.Nagata (eds.), {\em Topics in General Topology}
(Elsevier Sci. Publ., 1989), 315--366.

\bibitem[Ne]{Ned} S.~Nedev. {\em 0-metrizable spaces},
Trudy Moskov. Mat. Ob\v s\v c. {\bf 24} (1971) 201--236 (in
Russian).

\bibitem[Ny$_1$]{Nyi}
P.~Nyikos. {\em Metrizability and the Fr\'echet-Urysohn property
in topological groups}, Proc. Amer. Math. Soc. {\bf 83}:4 (1981),
793--801.

\bibitem[Ny$_2$]{Nyi2} P.J.~Nyikos. {\em Classes of compact sequential
spaces}. In: J.~Steprans, S.~Watson (eds.), {\em Set Theory and
its Applications}, Lect. Notes in Math. {\bf 1401}
(Springer-Verlag, Berlin etc., 1989), 135--159.

\bibitem[PZ]{PZ} I.~Protasov, E.~Zelenyuk, {\em Topologies on groups determined
by sequences}, Matem. Studii, Monograph Series, {\bf 4} (VNTL,
Lviv, 1999).



\bibitem[Sa]{Sa}
K. Sakai. {\em On $\mathbb{R}^\infty$-manifolds and
$Q^\infty$-manifolds}, Topology Appl. {\bf 18} (1984), 69--79.

\bibitem[Sch]{Sch} H.~Schaefer, {\em Topological Vector Spaces}  (The
Macmillan Company, NY, 1966).

\bibitem[Shi]{Shi} A.Shibakov, {\em Metrizability of sequential topological groups
with point-countable $ k $-network}, Proc. Amer. Math. Soc. {\bf
126} (1998), 943-947.


\bibitem[Tk]{Tk} M.Tkachenko, {\em Introduction to topological groups},
Topology Appl. {\bf 86} (1998), 179-231.

\bibitem[Va]{Va} J.E.~Vaughan. {\em Small uncountable cardinals and topology}, in:
J.~van Mill and G.M.Reed (eds.) {\em Open Problems in Topology}
(North-Holland, Amsterdam, 1990), 195--216.

\bibitem[Ya]{Ya}
N.N.~Yakovlev. {\em On the theory of $g$-metrizable spaces},
Soviet Math. Dokl. {\bf 17} (1976), 1217--1219. Russian original:
Dokl. Akad. Nauk SSSR {\bf 229} (1976), 1330--1331.

\bibitem[Zd]{Zd} L.~Zdomsky, {\em Interplay of topological and
algebraical structures on spaces with countable
$\mathrm{cs}$-character}, (in preparation).

\bibitem[Ze$_1$]{Ze} E.~Zelenyuk, {\em Topologies on groups determined by
compact subspaces}, Matem. Studii {\bf 5} (1995) 5-16 (in
Russian).

\bibitem[Ze$_2$]{Ze2} E.~Zelenyuk, {\em On group operations on homogeneous
spaces}, Proc. Amer. Math. Soc. {\bf 132}:4  (2004),  1219--1222.
\end{thebibliography}
\end{document}